\documentclass{amsart}
\usepackage{amsmath,amsthm,amsfonts,amssymb,tikz,scalerel,mathrsfs,marvosym}
\usetikzlibrary{matrix,arrows,decorations.pathmorphing}
\usetikzlibrary{shapes.geometric}
\usetikzlibrary{decorations.markings} % voor pijlen in midden van lijn

\tikzset{
  vertice/.style={circle,draw=black},
  decoration={markings,mark=at position 0.5 with {\arrow{>}}},
}

\usepackage[all]{xy}
\usepackage{enumerate}
\usepackage{blkarray}

\theoremstyle{plain}
\newtheorem{corollary}{Corollary}[section]
\newtheorem{convention}[corollary]{Convention}

\newtheorem*{theorem*}{Theorem}
\newtheorem{theorem}[corollary]{Theorem}
\newtheorem{lemma}[corollary]{Lemma}
\newtheorem*{lemma*}{Lemma}

\newtheorem{remark}[corollary]{Remark}
\newtheorem{proposition}[corollary]{Proposition}

\theoremstyle{definition}
\newtheorem{definition}[corollary]{Definition}
\newtheorem{definition*}{Definition}

\newcommand{\A}{\mathcal{A}}

\DeclareMathOperator{\Aut}{Aut}

\newcommand{\bl}[2]{\left\langle #1,#2\right\rangle}

\DeclareMathOperator{\BiMod}{BiMod}
\DeclareMathOperator{\bimod}{bimod}

\def\cd{\operatorname {cd}}

\DeclareMathOperator{\coh}{coh}
\DeclareMathOperator{\cone}{cone}
\DeclareMathOperator{\coker}{coker}
\renewcommand{\d}[1]{\mathbb{#1}}
\newcommand{\D}{{\mathscr{D}}}
\DeclareMathOperator{\der}{R}

\newcommand{\ds}{\oplus}
\newcommand{\DS}{\bigoplus}
\newcommand{\epi}{\xymatrix{{}\ar@{->>}[r]&{}}}
\DeclareMathOperator{\End}{End}
\DeclareMathOperator{\Ext}{Ext}
\newcommand{\f}[1]{\mathfrak{#1}}

\newcommand{\floor}[1]{\left\lfloor #1 \right\rfloor}

\newcommand{\fun}{\mapsto}

\DeclareMathOperator{\Gr}{Gr}

\let\H\relax %unassign \H
\DeclareMathOperator{\H}{H}

\DeclareMathOperator{\Hom}{Hom}
\DeclareMathOperator{\Id}{Id}
\DeclareMathOperator{\im}{im}

\renewcommand{\k}{\Bbbk}

\def\mod{\operatorname{mod}}
\newcommand{\mor}{\longrightarrow}

\newcommand{\T}{\mathcal{T}}
\renewcommand{\O}{\mathcal{O}}

\DeclareMathOperator{\Ob}{Ob}

\DeclareMathOperator{\Proj}{Proj}

\DeclareMathOperator{\Qcoh}{Qcoh}

\DeclareMathOperator{\RHom}{RHom}
\renewcommand{\r}[1]{\mathcal{#1}}
\newcommand\reallywidehat[1]{%
\begin{array}{c}
\stretchto{
  \scaleto{
    \scalerel*[\widthof{#1}]{\bigwedge}
    {\rule[-\textheight/2]{1ex}{\textheight}}
  }{1.25\textheight} % THIS STRETCHES THE WEDGE A LITTLE EXTRA WIDE
}{0.5ex}\\           % THIS SQUEEZES THE WEDGE TO 0.5ex HEIGHT
#1\\                 % THIS STACKS THE WEDGE ATOP THE ARGUMENT
\rule{-1ex}{0ex}
\end{array}
}
\DeclareMathOperator{\rk}{rk}

\def\ShExt{\mathcal{E}\mathit{xt}}

\DeclareMathOperator{\Spec}{Spec}
\DeclareMathOperator{\Supp}{Supp}

\DeclareMathOperator{\Sym}{Sym}
\def\ShHom{\operatorname {\mathcal{H}\mathit{om}}}

\renewcommand{\r}[1]{\mathcal{#1}}
\newcommand{\mono}{\xymatrix{{}\ar@{^{(}->}[r]&{}}}\DeclareMathOperator{\num}{num}

\DeclareMathOperator{\Tors}{Tors}
\newcommand{\tr}{\otimes}

\newcommand{\Z}{\mathbb{Z}}

\title[Homological Properties of a Noncommutative Del Pezzo Surface]{Homological Properties of a Certain Noncommutative Del Pezzo Surface}
\author{Louis de Thanhoffer de Volcsey \& Dennis Presotto}

\begin{document}
\begin{abstract}
In an upcoming paper, de Thanhoffer de Volcsey and Van den Bergh  consider a generalization of the numerical Grothendieck group of a Del Pezzo surface and show that if this group has an exceptional sequence of length 4, it must be of one of four types, the fourth one not coming from a commutative Del Pezzo surface.  In this paper, we adapt the theory of noncommutative $\d{P}^1$-bundles as appearing in the work of Van den Bergh and Nyman to produce a sheaf $\d{Z}$-algebra whose associated noncommutative projective scheme has a full exceptional sequence of length 4 for which the Gram matrix is of this fourth type. We show that this noncommutative scheme is noetherian and describe its local structure through the use of generalized preprojective algebras as defined in (\cite{Genpreprojective})
\end{abstract}
\maketitle

\tableofcontents
\section{Introduction and Statement of Results}

In the paper, \cite{deTVdB14}, de Thanhoffer de Volcsey and Van den Bergh provide a numerical classification of possibly noncommutative Del Pezzo surfaces with an exceptional sequence of length 4. More precisely they consider a free abelian group $\Lambda$ with a nondegenerate bilinear form $\langle -,-\rangle $ and consider the following sets of conditions

\begin{itemize}
\item there is an $s \in \Aut(\Lambda)$ such that $\langle x,sy\rangle=\langle y,x \rangle$ for $x,y \in \Lambda$
\item $(s-1)$ is nilpotent
\item $\rk(s-1)=2$
\item $\bl{(s-1)x}{(s-1)x}<0$ for $x \notin \ker(s-1)$
\item the form is indefinite on a certain subquotient of $\Lambda$
\end{itemize}
It is proved that the numerical Grothendieck group $K(X)_{\num}$, together with the Euler form, of a Del Pezzo surface $X$ satisfies these conditions. The classification result they obtain is the following:
\begin{theorem*}
Let $\Lambda$ satisfy the above conditions. Then $\Lambda$ is isomorphic to $\d{Z}^4$ where the matrix of the bilinear form is one of the following standard types:
$$\begin{bmatrix}
1&0&0&0\\
0&1&3&6\\
0&0&1&3\\
0&0&0&1\\
\end{bmatrix},
\begin{bmatrix}
1&2&2&4\\
0&1&0&2\\
0&0&1&2\\
0&0&0&1\\
\end{bmatrix},
\begin{bmatrix}
1&2&3&5\\
0&1&1&3\\
0&0&1&2\\
0&0&0&1\\
\end{bmatrix}
\textrm{ and  }
\begin{bmatrix}
1&2&1&5\\
0&1&0&4\\
0&0&1&2\\
0&0&0&1\\
\end{bmatrix}
$$
\end{theorem*}
The first type is a trivial extension the Grothendieck group of $\mathbb{P}^2$ to a rank 4 lattice. The matrices of type 2 and 3 correspond to the Grothendieck groups
of the Hirzebruch surfaces $\mathbb{P}^1 \times \mathbb{P}^1$ and $\mathbb{F}_1$ respectively. Moreover, it was proven in \cite{deTVdB14} that the last type does not correspond to the Grothendieck group of a Del Pezzo surface.\\ \\
The goal of this paper is to construct a noncommutative surface $Z$ together with a full exceptional sequence $E$ of $Z$-modules whose classes in $K(Z)$ form a basis in which the Euler form has the desired Gram matrix. As the top-left and bottom-right $2\times 2$ submatrices, show that the sequences $(E_1,E_2)$  and $(E_3,E_4)$ are isomorphic to the standard sequence $(\O_X,\O_X(1))$ on $K(\d{P}^1)$, we heuristically conclude that $Z$ should be equipped with 2 `maps' (in the noncommutative sense) $\Pi_0,\Pi_1: Z\mor \d{P}^1$ such that $E$ is obtained by pulling back $(\O_X,\O_X(1))$ along both:
\begin{gather}
\label{eq:exceptionalsequence}
E=\bigg(\Pi_1^*(\r{O}_{\d{P}^1}),\Pi_1^*(\r{O}_{\d{P}^1}(1)), \Pi_0^*(\r{O}_{\d{P}^1}),\Pi_0^*(\r{O}_{\d{P}^1}(1))\bigg)
\end{gather}
The construction of this noncommutative surface is an adaptation of Van den Bergh's theory of noncommutative $\d{P}^1$-bundles over a smooth base scheme $X$ of finite type over $\k$ as developed in \cite{VdB_12}. In that paper Van den Bergh proposes a new construction which results in a sheafified notion of a $\Z$-algebra: more precisely, let $\r{E}$ be a \emph{coherent} $X$-bimodule (see Definition \ref{def:bimodule}) which is locally free on both sides. Then there is an appropriate notion of left- and right dual $^{*}{\r{E}}$ (resp. $\r{E}^*$) in this context (Lemma \ref{lem:dual}). Applying the construction indefinitely yields $^{*m}{\r{E}}$ resp. $\r{E}^{*m}$, which by naturality comes with a unit morphism
\[
i_m:\O_\Delta \mor \r{E}^{*m} \tr_X \r{E}^{*m+1}
\]
Van den Bergh defines the  \emph{symmetric sheaf $\d{Z}$-algebra}  $\d{S}(\r{E})$ as 
\begin{itemize}
\item $\d{S}(\r{E})_{m,m}=\r{O}_X$
\item $\d{S}(\r{E})_{m,m+1}=\r{E}^{*m}$
\item $\d{S}(\r{E})$ is freely generated by $\d{S}(\r{E})_{m,m+1}$ subject to the relations given by the images of the morphism $i_m$	
\end{itemize}
(see Definition \ref{def:nodig}). There is an associated category of graded $\A$-modules $\Gr(\d{S}(\r{E}))$ which is Grothendieck (Theorem \ref{thm:groth}). The intuition behind the definition of $\d{S}(\r{E})$ comes from the fact that in the case where $\r{E}$ is central of rank $(2,2)$, the definition coincides with the notion of a $\d{P}^1$-bundle over $X$ in the sense that there is an equivalence  between their categories of graded modules. For the convenience of the reader, we provide an explicit proof of this in Corollary \ref{cor:commutativealgebra}.\\[\medskipamount]
Pushing our heuristic intuition further, 
\[
\dim_\k \left(\Hom_Z(\Pi_1^*(\r{O}_{\d{P}^1}(1)),\Pi_0^*(\r{O}_{\d{P}^1}(1)) \right)=\bl{\Pi_1^*\r{O}_{\d{P}^1}(1))}{\Pi_0^*(\r{O}_{\d{P}^1}(1)}=4
\]
 seems to suggest that $\r{E}$ should be built using a morphism of degree $4$ on the left and  similarly by the identity on the right. This leads one to adapt Van den Bergh's construction under the additional assumptions that the bimodule $\r{E}$ is locally free of rank $(4,1)$ over a pair of base schemes $X,Y$. 
To construct the noncommutative scheme $\Proj(\d{S}(\r{E}))$ and establish its properties, we shall first prove two facts in the setting: 
The first is a description of $\d{S}(\r{E})$ in the case where the base scheme $X$ is affine. More precisely,  we relate $\d{S}(\r{E})$ to the generalized preprojective algebras $\Pi_R(S)$ associated to a morphism of rings $R\mor S$ which is relatively Frobenius as introduced in \cite{Genpreprojective} as follows:
\begin{theorem*}
\label{maintheorem:localdescription}(see \ref{thm:coverconditions} and \ref{lem:catC2} together with \ref{lem:locallypirs})
Let $\r{E}$ be an $X$-$Y$-bimodule of rank $(4,1)$. Then there is a finite affine open cover $U_i\subset X$ such that
	the category  $\Gr(\d{S}(\r{E})\arrowvert_{U_i})$ identifies with a direct summand of $\Gr(\Pi_{R_i}(S_i))$ where $R_i\mor S_i$ is relatively Frobenius of rank 4
\end{theorem*}
Second, we adapt the technique of point modules which was developed in \cite{VdB_12} for the rank $(2,2)$ to the rank $(4,1)$ case. This proves to be a substantial modification, requiring an adaptation of the very definition of point module. We use this technique to prove that $\d{S}(\r{E})_{n,m}$ is a locally free bimodule in each degree. This, together with the previous result allows us to adapt the ideas of \cite{Mori07} and \cite{Nyman04} to obtain local noetherianity of the category $\Gr(\d{S}(\r{E}))$. This allows us in turn to consider the noncommutative scheme $Z=\Proj(\d{S}(\r{E}))$ in the language of \cite{ArtZhang94}. We summarize:
\begin{theorem*}(see \ref{thm:noeth} and \ref{cor:locallyfree})
	Let $\r{E}$ be an $X$-$Y$-bimodule of rank $(4,1)$. Then
	\begin{itemize}
		\item
		$Z:=\Proj(\d{S}(\r{E}))$ is noetherian. 
		\item 
		Each bimodule $\d{S}(\r{E})_{n,m}$ is locally free and the ranks can be explicitly computed\footnote{Moreover these ranks agree with the numbers obtained in \cite{Nyman15} where this case was studied for $X,Y$ zero-dimensional schemes}. In particular, if $n-m$ is even, these ranks coincide with the ``classical'' case where $\r{E}$ has rank $(2,2)$. (See Corollary \ref{cor:locallyfree}) .	
	\end{itemize}
\end{theorem*}
The noncommutative scheme $Z$ comes with a sequence of maps $\Pi_{2n}:Z\mor X$, $\Pi_{2n+1}: Z \mor Y$ (again in the sense of \cite{ArtZhang94}) given by taking the corresponding degree of the graded module. We describe these and show that $\Pi_0$ and $\Pi_1$ contain all the information on these maps in a certain sense. With these definitions, we finally prove
\begin{theorem*}(See \ref{thm:mainexoticsequence})
Let  $\r{E}$ be a $\d{P}^1$-bimodule of rank $(4,1)$. Let $\d{S}(\r{E})$ be the associated symmetric sheaf $\d{Z}$-algebra and put $Z=\Proj(\d{S}(\r{E}))$. Then 
\[
\bigg( \Pi_1^*(\r{O}_{\d{P}^1}),\Pi_1^*(\r{O}_{\d{P}^1}(1)), \Pi_0^*(\r{O}_{\d{P}^1}),\Pi_0^*(\r{O}_{\d{P}^1}(1))\bigg)
\]
is a full strong exceptional sequence on $Z$ for which the Gram matrix of the Euler form is given by
\[
\begin{bmatrix}
1&2&1&5\\
0&1&0&4\\
0&0&1&2\\
0&0&0&1\\
\end{bmatrix}
\]
\end{theorem*}

\section{Symmetric Sheaf $\d{Z}$-Algebras}

\subsection{Definitions and Construction}
We begin by giving a summary of the material needed to construct symmetric sheaf $\Z$-algebras following \cite{VdB_12}. 
\begin{convention}
\label{conv:conditionsxy}
Throughout the paper $\k$ denotes an algebraically closed field. $W$, $X$ and $Y$ will denote smooth varieties (that is smooth, integral\footnote{One could leave out this condition, which leads to the more general setting of disjoint unions of varieties, we choose not to do this for purposes of clarity}, separated and of finite type over $\k$).

\end{convention}

\begin{definition} \label{def:bimodule}
A coherent $X-Y$ bimodule $\r{E}$ is a coherent $\r{O}_{X\times Y}$-module such that the support of $\r{E}$ is finite over $X$ and $Y$. We denote the corresponding category by $\bimod(X-Y)$. More generally an $X-Y$-bimodule is a quasi-coherent $\r{O}_{X\times Y}$-module which is a filtered direct limit of objects in $\bimod(X-Y)$. The  category of $X-Y$-bimodules is denoted $\BiMod(X-Y)$. Finally, a bimodule $\r{E}$ is called locally free if ${\pi_X}_*(\r{E})$ and ${\pi_Y}_*(\r{E})$ are locally free, where $ \pi_X, \pi_Y$ denote the standard projections. If $\rk ({\pi_X}_*(\r{E}))=m$ and $\rk({\pi_Y}_*(\r{E}))=n$, we write $\rk\r{E}=(m,n)$.
\end{definition}

For $W$, $X$ and $Y$ the tensor product of $\r{O}_{W \times X \times Y}$-modules induces a tensor product
\begin{equation*}
\label{def:tensorproductbimods}
\BiMod(W-X) \tr  \BiMod(X-Y) \mor \BiMod(W-Y):(\r{E},\r{F})\mapsto \r{E}\tr_X \r{F}
\end{equation*}
through the formula
\[ \r{E} \tr \r{F} :=  {\pi_{W\times Y}}_*\big(\pi^*_{W\times X}(\r{E}) \tr_{W \times X \times Y} \pi^*_{X\times Y} ( \r{F}) \big) \]
For each $\r{E} \in \BiMod(W-X)$ this defines a functor :
\begin{equation}
\label{def:tensorproductwithbimod}
 - \tr_X \r{E}: \Qcoh(W) \mor \Qcoh(X):\r{M}\fun  \r{M} \tr_X \r{E} := {\pi_X}_*\big(\pi^*_W(\r{M}) \tr_{W \times X} \r{E} \big)
\end{equation}
which is right exact in general and  exact if $\r{E}$ is locally free on the left. We mention that \cite[lemma 3.1.1.]{VdB_12} shows that this functor determines the bimodule $\r{E}$ uniquely.

\begin{definition}
\label{def:standardformbimodule}
Consider morphisms ${u: W \mor X}$, ${v: W \mor Y}$. If $\r{U} \in \Qcoh(W)$, then we denote $(u, v)_*\r{U} \in \BiMod(X-Y)$
as $_u\r{U}_v$. One easily checks:
\begin{equation}
\label{eq:pullbacktensorproduct}
 - \tr {}_u\r{U}_v = v_*(u^*(-) \tr_W \r{U})
 \end{equation}
A bimodule isomorphic to one of the form ${}_u \r{U}_u \cong {}_{\Id} ( u_* \r{U} )_{\Id}$ is called \emph{central}.
\end{definition}

Next we introduce the language of sheaf $\d{Z}$-algebras, a `sheafified version' of a classical $\d{Z}$-algebra.

\begin{definition} \label{def:nodig}
Let $(X_i)_{i\in \d{Z}}$ be a sequence of smooth varieties.\\
A sheaf $\d{Z}$-algebra $\r{A}$, is a collection of $X_i-X_j$-bimodules $\r{A}_{ij}$  equipped with multiplication - and identity maps
\[
\mu_{i,j,k}:\r{A}_{i,j}\tr \r{A}_{j,k} \mor \r{A}_{i,k} \text{ and }
u_i:\r{O}_{X_i}\mor \r{A}_{i,i}
\]
such that the usual associativity 
%\begin{displaymath}
%\xymatrix{
%\r{A}_{i,j}\tr \r{A}_{j,k}\tr{A}_{k,l} \ar[rr]^{\mu_{i,j,k}\tr 1}\ar[d]_{1\tr \mu_{j,k,l}} & & \r{A}_{i,k}\tr \r{A}_{k,l}\ar[d]^{\mu_{i,k,l}}\\
%\r{A}_{i,j}\tr \r{A}_{j,l}\ar[rr]_{\mu_{i,j,l}} & & \r{A}_{i,l}
%}	
%\end{displaymath}
and unit diagrams
%\begin{displaymath}
%\xymatrix{
%\r{O}_{X_i}\tr \r{A}_{i,j} \ar[rr]^{u_i\tr 1}\ar[dr] && \r{A}_{i,i}\tr \r{A}_{i,j}\ar[dl]^{\mu_{i,i,j}}&\r{A}_{i,j}\tr \r{O}_{X_i} \ar[rr]^{1\tr u_i}\ar[dr] && \r{A}_{i,j}\tr \r{A}_{j,j}\ar[dl]^{\mu_{i,j,j}}\\
%& \r{A}_{i,j}&&&\r{A}_{i,j}}
%\end{displaymath}
commute.
 \end{definition}
In a similar vain, we introduce the notion of graded (right) module over a sheaf $\Z$-algebra:
\begin{definition} Let $\r{A}$ be a sheaf $\Z$-algebra.\\
A graded $\r{A}$-module is a sequence of $X_i$-modules $\r{M}_i$ together with maps 
\[
\mu_{i,j}:\r{M}_i \tr \r{A}_{i,j} \mor \r{M}_j
\]

compatible with the multiplication and identity maps on $\r{A}$ in the usual sense.
A morphism of graded  (right) $\r{A}$-modules $f:\r{M}\mor \r{N}$ is a collection of $X_i$-module morphisms $f_i:\r{M}_i\mor \r{N}_i$ such that the diagram
\begin{displaymath}
\xymatrix{
\r{M}_i\tr \r{A}_{i,j}\ar[rr]^{f_i}\ar[d] &&\r{N}_i\tr \r{A}_{i,j}\ar[d] \\
\r{M}_j\ar[rr]_{f_j}&& \r{N}_j
}	
\end{displaymath}
commutes. The associated category is denoted $\Gr(\r{A})$.
\end{definition}

\begin{definition}
\label{def:proj}
An $\r{A}$-module is right bounded if $\r{M}_i = 0$ for $i \gg 0$. An $\r{A}$-module is called \emph{torsion} if it is a filtered colimit of right bounded modules. Let $\Tors(\r{A})$ be the subcategory of $\Gr(\r{A})$ consisting of torsion modules. If $\Gr(\r{A})$ is a locally noetherian category\footnote{As mentioned in the introduction, this property is nontrivial and in fact one of the main results of this paper in case $\r{A}$ is a \emph{symmetric sheaf $\mathbb{Z}$-algebra}.}, $\Tors(\r{A})$ is a localizing subcategory and the corresponding quotient category is denoted by $\Proj(\r{A})$.\\
This construction yields a projection functor $p: \Gr(\r{A})\mor \Proj(\r{A})$ with right adjoint $\omega$ (see \cite{Smith_99}).
\end{definition}

\begin{remark}
It an easy observation that $\Gr(A)$ is abelian and that all universal constructions are defined `degreewise'	
\end{remark}

The fundamental example of a graded right $\r{A}$-module is given by the collection $e_n\r{A}$ satisfying
\begin{equation}
\label{def:e_nA}
\big(e_n\r{A}\big)_i=\r{A}_{n,i}	
\end{equation}
The first crucial step in our construction is a certain duality between locally free bimodules. To this end, we recall that by Convention\ref{conv:conditionsxy}, $X$ and $Y$ denote smooth varieties over $\k$.

\begin{lemma}(see \cite[\S4]{VdB_12})
\label{lem:dual}
Let $\r{E} \in \bimod(X-Y)$ be a locally free coherent bimodule. Then there is a unique object $\r{E}^* \in \bimod(Y-X)$ such that the functor
\[ 
- \tr_Y \r{E}^*: \Qcoh(Y) \mor \Qcoh(X)
\]
(see (\ref{def:tensorproductwithbimod})) is  right adjoint to the functor $- \tr_X \r{E}$., i.e for $\r{M} \in \Qcoh(X)$ and ${\r{N} \in \Qcoh(Y):}$
\[ \Hom_Y(\r{M}\tr \r{E}, \r{N}) \cong \Hom_X(\r{M}, \r{N} \tr \r{E}^*) \]
\end{lemma}

\begin{remark}\label{rem:dual} Van den Bergh also gives an explicit formula (see the discussion following Proposition 4.1.6 in \cite{VdB_12}).
If $\r{E} = {}_u \r{U}_v$ then $\r{E}^*$ is given by ${}_v \ShHom_W(\r{U}, v^! \r{O}_Y)_u$
\end{remark}
The dual notion leads to the left dual: an object $^{*}{\r{E}}$ such that
\[ \Hom_X(\r{N}\tr ^{*}{\r{E}}, \r{M}) \cong \Hom_Y(\r{N}, \r{M} \tr \r{E})\]
By Yoneda's lemma we have 
\begin{equation} \label{eq:yoneda} \r{E}=^{*}{(\r{E}^*)}=(^{*}{\r{E}})^* \end{equation}
Repeated application of  duals leads to the following notation:
$$\r{E}^{*n}=
\begin{cases}
\r{E}^{\overbrace{*\ldots *}^{n}} & n \ge 0\\
^{\overbrace{*\ldots *}^{-n}}{\r{E}} & n<0
\end{cases}
$$
In the sequel  it will be convenient to invoke the following notation:
\begin{convention}\label{conv:Z_n} For $X$ and $Y$, we shall without further mention consider the sequence $(X_n)_{n \in \Z}$ defined as
\[
 X_n=X \text{ if } n \text{ is even and  }Y \text{ if } n \text{ is odd}
\] 
 \end{convention}
From the adjointness properties of the duals defined above, there are unit and counit morphisms:
\begin{align}
\label{eq:unitmorphism} i_n:&  \ \r{O}_{X_n} \mor \r{E}^{*n} \tr \r{E}^{*n+1} \\
\notag j_{n}:& \ \r{E}^{*n}\tr \r{E}^{*n-1} \mor \r{O}_{X_n} \end{align}
Our next ingredient is that of a nondegenerate bimodule.
\begin{definition}
We say that $\r{Q}\in \bimod(X-W)$ is \emph{invertible} if there exists a bimodule $\r{Q}^{-1} \in \bimod(W-X)$ such that 
\[
\r{Q} \tr_W \r{Q}^{-1} \cong \r{O}_X \text{ and } \r{Q}^{-1} \tr_X \r{Q} \cong \r{O}_W.
\]
 If there exist $\r{E} \in \bimod(X-Y)$ and $\r{F} \in \bimod(Y-W)$ such that $\r{Q}\subset \r{E}\tr_Y \r{F}$, we say the inclusion is \emph{nondegenerate} if the following composition 
\[ 
\r{E}^*\tr_X \r{Q}\mor \r{E}^*\tr_X \tr \left( \r{E} \tr_Y \r{F} \right) \mor \r{F}
\]
is an isomorphism.
\end{definition}

\begin{definition}
Let $(X_i)_{i \in \Z}$ be a sequence of smooth varieties over $\k$ and let $\r{E}_i$ be locally free $X_i-X_{i+1}$-bimodules. Then the \emph{tensor sheaf $\d{Z}$-algebra} $\d{T}(\{\r{E}_i\})$ is the sheaf $\d{Z}$-algebra generated by the $ \{\r{E}_i\}$,  more precisely
\[ \d{T}(\{\r{E}_i\})_{m,n} = \begin{cases}
0 & n<m\\
{}_{\Id} \big( \r{O}_{X_m} \big)_{\Id} & n=m\\
\r{E}_m \tr \ldots \tr \r{E}_{n-1} & n>m\\
\end{cases}
\]
If we are given smooth varieties $X$ and $Y$ and a locally free $X-Y$-bimodule $\r{E}$, the standard tensor algebra is the sheaf $\Z$-algebra $\d{T}(\r{E})$ constructed by applying the convention \ref{conv:Z_n} and defining 
\[
\r{E}_n  =  \r{E}^{*n}
\]
to the above definition
\end{definition}
We can now state the definition of the main object of study in this paper: the symmetric sheaf $\Z$-algebra.
\begin{definition}
Let $(X_i)_{i \in \d{Z}}$ be a sequence of smooth varieties over $\k$ and let $\r{E}_i$ be locally free $X_i-X_{i+1}$-bimodules. Suppose that for each $i$ we are given a nondegenerate $X_i-X_{i+2}$-bimodule ${\r{Q}_i \subset \r{E}_i \tr\r{E}_{i+1}}$. Then the \emph{symmetric sheaf $\d{Z}$-algebra} $\d{S}(\{\r{E}_i\}, \{ \r{Q}_i \})$ is the quotient of $\d{T}(\{\r{E}_i\})$ by the relations $(\r{Q}_i)_i$. More precisely, $\d{S}(\{\r{E}_i\}, \{ \r{Q}_i \})_{m,n}$ is defined as
\[ \begin{cases}
\d{T}(\{\r{E}_i\})_{m,n} & n \leq m+1\\
\d{T}(\{\r{E}_i\})_{m,n} / \bigg((\r{Q}_m\tr\ldots) + (\r{E}_m\tr Q_{m+1} \tr \ldots)+\ldots + (\ldots \tr Q_{n-2})\bigg) & n \geq m+2
\end{cases}
\]
If $X$ and $Y$ are smooth varieties, and $\r{E}$ an $X-Y$-bimodule,  the standard symmetric sheaf $\d{Z}$-algebra $\d{S}(\r{E})$ is constructed by considering the standard tensor algebra $\d{T}(\r{E})$ and considering the following sequence of nondegenerate invertible bimodules:
\begin{equation}
\label{eq:standardsheafZalgebra}
\r{Q}_n =i_n \left( \r{O}_{X_n} \right) \subset \r{E}^{*n}\tr \r{E}^{*n+1}
\end{equation}

\end{definition}

A fundamental operation in the context of sheaf $\Z$-algebras is that of twisting by a sequence of invertible bimodules:
\begin{theorem}\label{thm:twisting}
Let $(X_i)_i$ and $(Y_i)_i$ be sequences of smooth varieties over $\k$ and $\r{A}$ a sheaf $\d{Z}$-algebra on $(X_i)_i$.\\ 
Given a collection of invertible $X_i-Y_i$-bimodules $(\r{T}_i)_i$, one can construct a sheaf $\d{Z}$-algebra $\r{B}$ by
\[ \r{B}_{ij}:=\r{T}_i^{-1} \tr \r{A}_{ij}\tr \r{T}_j \]
called the twist of $\r{A}$ by $(\r{T}_i)_i$.\\
There is an equivalence of categories given by the functor 
\[ \r{T}: \Gr(\r{A})\cong \Gr(\r{B}): \r{M}_i \mor \r{M}_i \tr \r{T}_i \]
Moreover, every symmetric sheaf $\d{Z}$-algebra can be obtained from a standard symmetric one by a twist.
\end{theorem}
\begin{proof}
This is proven in \cite[\S 5.1]{VdB_12}	
\end{proof}
We also have the following important result concerning graded modules over symmetric sheaf $\mathbb{Z}$-algebras:
\begin{theorem} \label{thm:groth}
Let $\r{A}$ be a symmetric sheaf $\d{Z}$-algebra. Then $\Gr(\r{A})$ is Grothendieck.
\end{theorem}

\begin{proof}
Let $(\r{M}_i,f_{ij})$ be a direct system of graded $\r{A}$-modules. In each degree $d$, we obtain a direct system of quasi-coherent $X_d$-modules $(\r{M}^d,f_{ij}^d)$. Since $\Qcoh(X_n)$ is Grothendieck, we can form the direct limit in each degree to obtain a sequence of $X_n$-modules $\r{L}_n:=\varinjlim (\r{M}_i^n,f_{ij}^n)$. If we fix a couple $(n,m)$, the universality of the direct limit naturally defines a map 
\[ \r{A}_{n,m} \tr \r{L}_n = \r{A}_{n,m} \tr \varinjlim (X_i^n,f_{ij}^n) \mor \varinjlim (X_i^m,f_{ij}^m) = \r{L}_m \]
 showing that $\r{L}$ is in fact a graded $\r{A}$-module. The fact that $\r{L}$ is a direct limit and that the formation of $\r{L}$ is exact is an easy consequence of the construction.\\
 Next, for each, $i$ let $\r{G}^j_{i}$ be a collection of generators for $\Qcoh(X_i)$. Then the collection 
\[ \{ \r{G}^j_i \tr e_i \r{A} \mid n \in \d{Z}, \r{N} \in \r{N}^n \} \]
forms a set of generators for $\Gr(\r{A})$. 
\end{proof}

\subsection{The Rank $(2,2)$ Case}
\label{subsec:rank(2,2)case}
In this section, we give a proof of the result that $\Proj(\d{S}(\r{E}))$ is Morita equivalent to a commutative scheme in the case where $\r{E}$ has rank $(2,2)$. As mentioned in the introduction however, our primary concern is the rank $(4,1)$ case. This section's sole purpose is to acquire a little geometric intuition in $\d{S}(\r{E})$ and as such  may be skipped by the reader without any trouble.\\
We first begin by introducing notation for the $\d{Z}$-graded-to-$\d{Z}$-algebra construction in our setting:
\begin{convention}
\label{conv:hat}
Let $\DS_n\r{G}_n$ be a graded algebra in the monoidal category $\bimod(X)$ Then we denote by $\widehat{\r{G}}$ the sheaf $\Z$-algebra over $X$ whose $(i,j)$-component is the $X$-bimodule $\r{G}_{j-i}$.
\end{convention}

\begin{remark}
\label{rem:hatequivalence}
It is clear that in the above situation, taking the direct sum yields an equivalence:
\[
\Gr(\r{G})\stackrel{\simeq}{\mor} \Gr(\widehat{\r{G}}): (\r{M})_i\fun \bigoplus_i \r{M}_i
\]
\end{remark}

The following lemma (which was already announced but not proven in \cite{VdB_12}) shows that  symmetric sheaf $\d{Z}$-algebras over central bimodules rank (2,2) indeed essentially coincide with sheaves of commutative graded algebras:
\begin{lemma}\label{lem:commutativealgebra}
Let $\r{V}$ be a locally free $X$-module of rank 2. There is an equivalence of the form
\[ \Gr(\d{S}(_{\Id}\r{V}_{\Id}))\stackrel{\r{T}}{\mor} \Gr \big( \reallywidehat{\textrm{\emph{Sym}$_{X\times X} (_{\Id}\r{V}_{\Id})$}}\big)\stackrel{\simeq}{\mor}\Gr(\Sym_X(\r{V})) \]
where $\r{T}$ is given by twisting through $\big(\big(\bigwedge^2\r{V}\big)^{\floor{\frac{i}{2}}}\big)_{i\in \d{Z}}$.
\end{lemma}

\begin{proof}
We first describe the second equivalence. By Remark \ref{rem:hatequivalence}, we may remove the hat and simply consider the sheaf of graded algebras  $\Sym_{X\times X} (_{\Id}\r{V}_{\Id})$. The second equivalence now follows tautologically from the definitions, since in each degree $d,d'$,
\[ \r{M}_d \tr \textrm{Sym}_{X \times X}( {}_{\Id} \r{V} _{\Id})_{d'} = \r{M}_d \tr {}_{\Id} \left( \textrm{Sym}_X( \r{V} ) \right)_{\Id} \stackrel{(\ref{eq:pullbacktensorproduct})}{=} \r{M}_d \tr_X \textrm{Sym}_X( \r{V} )_{d'} \]
implying that both multiplications coincide. We now explain the first equivalence:\\
Let $\r{E}={}_{\Id}{\r{V}}_{\Id}$. Using the explicit expression for the dual given in Remark \ref{rem:dual}, we obtain
\[ \r{E}^*={}_{\Id}\r{H}om(\r{V},\Id^!\r{O}_X)_{\Id}={}_{\Id}( \r{V}^*)_{\Id} \]
In particular the equalities $\r{E}^{*2n} = \r{E} = {}_{\Id}( \r{V})_{\Id}$ and $\r{E}^{*2n+1}=\r{E}^*={}_{\Id}( \r{V}^*)_{\Id}$ hold for all $n$.
Since the pairing $\r{V}\tr \r{V} \mor \Lambda^2\r{V}$ is perfect,  there is an isomorphism
\begin{equation} \label{eq:exterior} \r{V}^* \tr (\Lambda^2 \r{V}) \stackrel{\cong}{\mor} \r{V}
\end{equation}
Let $(\r{T}_i)_i=(\bigwedge^{2}\r{V})^{\floor{\frac{i}{2}}}$. It follows from the definition of  $\d{T}(\r{E})$, that as sheaf $\d{Z}$-algebras, we have 
\[
\d{T}(\r{E})=\widehat{T_X(\r{V})}
\]
By Theorem \ref{thm:twisting}, applying the twist by the sequence $(\T_i)$ yields an equivalence
\begin{equation*}
\Gr(\d{T}(\r{E})) \rightarrow \Gr( \widehat{T_X(\r{V})}): \big( \r{M}_i \big)_i \mapsto \big( \r{M}_i \tr (\Lambda^2\r{V})^{\floor{\frac{i}{2}}}\big)_i \end{equation*}
specifically in each component:
\begin{equation} \label{eq:tensor1} \d{T}(\r{E})_{m,n} \cong {}_{\Id}\left( (\Lambda^2\r{V})^{\floor{\frac{m}{2}}} \tr T_X(\r{V})_{n-m} \tr (\Lambda^2\r{V})^{-\floor{\frac{n}{2}}} \right) _{\Id}
\end{equation}
We now claim that the twisting in (\ref{eq:tensor1}) induces a twisting
\begin{equation*} %\label{eq:symm1} 
\d{S}(\r{E})_{m,n} \cong {}_{\Id}\left( (\Lambda^2\r{V})^{\floor{\frac{m}{2}}} \tr \Sym_X(\r{V})_{n-m} \tr (\Lambda^2\r{V})^{-\floor{\frac{n}{2}}} \right)_{\Id}
\end{equation*}
and hence an equivalence of categories:
\begin{equation} \label{eq:sym2}\Gr(\d{S}(\r{E})) \rightarrow \Gr( \Sym_X(\r{V})): \big( \r{M}_i \big)_i \mapsto \bigoplus_i \r{M}_i \tr (\Lambda^2\r{V})^{\floor{\frac{i}{2}}} \end{equation}
So we are left with proving the claim. For this we must understand what happens under (\ref{eq:tensor1}) to the relations that define $\d{S}(\r{E})$ as a quotient of $\d{T}(\r{E})$.\\
As the relations are generated in degree 2 it suffices to consider ${\d{S}(\r{E})_{m,m+2} \tr {}_{\Id} (\Lambda^2\r{V})_{\Id}}$. This is the quotient of $\d{T}(\r{E})_{m,m+2} \tr {}_{\Id} (\Lambda^2\r{V})_{\Id} \cong {}_{\Id} \left( {T_X(\r{V})_2} \right)_{\Id} = {}_{\Id} \left( \r{V} \tr \r{V} \right)_{\Id}$ by the relation $i\big( {}_{\Id}( \r{O}_X)_{\Id} \big) \tr {}_{\Id} (\Lambda^2\r{V})_{\Id} \subset {}_{\Id} \left( \r{V} \tr \r{V}^* \tr \Lambda^2\r{V} \right) _{\Id} \cong  {}_{\Id} \left( \r{V} \tr \r{V} \right)_{\Id}$.
We have to check that this relation is exactly the one that defines $\Sym_X(\r{V})$ as a quotient of $\textrm{T}_X(\r{V})$. The latter relation is defined locally, so it suffices to check on a trivializing open subset $U$ for $\r{V}$. If $\left.\r{V} \right|_U \cong \left.\r{O}_X\right|_U u \oplus \left.\r{O}_X\right|_U v$ then $i\big( {}_{\Id}( \r{O}_X)_{\Id} \big)$ is locally given by $u \tr u^* + v \tr v^*$. One checks that the isomorphism (\ref{eq:exterior}) maps $u^* \tr (u \wedge v)$ to $v$ and $v^* \tr (u \wedge v)$ to $-u$, the induced relation in $\r{V} \tr \r{V}$ is locally given by $u \tr v - v \tr u$, the defining relation of $\Sym_X(\r{V})$.
\end{proof}

We have the following result:
\begin{proposition}
	Let $\r{E}$ be any $X-Y$-bimodule of rank $(2,2)$. Then $\Gr(\d{S}(\r{E}))$ is a locally noetherian category.
\end{proposition}
\begin{proof}
This is \cite[Theorem 1.2]{VdB_12}.
\end{proof}
This above proposition ensures that we can perform the $\Proj$ construction on $\d{S}(\r{E})$ if the rank of $\r{E}$ is (2,2). The resulting noncommutative scheme is equivalent a projective bundle over $X$ as follows:

\begin{corollary}\label{cor:commutativealgebra}
Let $\r{V}$ be locally free of rank 2 as above, then we have an induced equivalence:
\[ \Phi: \Proj(\d{S}({}_{\Id}(\r{V})_{\Id})) \stackrel{\cong}{\mor} \Proj(\Sym_X(\r{V})) \stackrel{\cong}{\mor} \Qcoh(\d{P}_X(\r{V})) \]
\end{corollary}

\begin{proof}
The equivalence given in (\ref{eq:sym2}) obviously maps torsion modules onto torsion modules. Hence, it yields an equivalence  $\Proj(\d{S}({}_{\Id}(\r{V})_{\Id}) \stackrel{\cong} {\mor} \Proj( \Sym_X(\r{V}))$.\\
The second equivalence is a well known result from classical algebraic geometry and is given by the following pair of functors
\[
\begin{tikzpicture}
\node (b)at (0,0) {$\Proj(\Sym_X(\r{V}))$};
\node (c) at (4,0) {$\Qcoh(\d{P}_X(\r{V}))$};
\node (d) at (2,1.2) {$\widetilde{(-)}$};
\node (e) at (2,-1.1) {$p \circ \Gamma_* := p \left[ \ds_i \pi_*\big( (-)(i)\big)\right]$};
\draw[->] (b) to[bend left](c) ;
\draw[->] (c) to[bend left] (b);
\end{tikzpicture} \]
Where $\pi$ is the canonical projection $\pi: \d{P}_X(\r{V}) \mor X$. \qedhere 

\end{proof}

\subsection{Truncation Functors and Periodicity}
\label{subsec:truncation}
Let $\r{A}$ be a sheaf $\Z$-algebra over a sequence of varieties $(X_i)_{i \in \Z} $. We define a sequence of \emph{truncation} functors as follows: for each $m \in \d{Z}$, we can consider the functor
\begin{displaymath}
\xymatrix{
 \Gr(\r{A})\ar[rr]^{(-)_m} &&  \Qcoh (X_m)
 }
 \end{displaymath}

We shall need the following easy result on these functors:
\begin{lemma}\label{lem:adjunctiontruncation}
Let $e_m\r{A}$ be the right $\r{A}$-module defined in \ref{def:e_nA} . There is an adjoint pair
\[ - \tr e_m\r{A} \ \dashv \ (-)_m \]
\end{lemma}
\begin{proof}
The proof of this is standard and left to the reader
\end{proof}
Our next result shows that there is a certain 2-periodic behavior among these functors. To this end, for $n \in \Z$, we denote by $\r{A}(n)$ the sheaf $\Z$-algebra
\begin{equation}
\label{eq:twistingsheafzalgebra}
\r{A}(n)_{i,j}=\r{A}_{n+i,n+j}
\end{equation}

\begin{proposition}\label{prop:2-periodicity}
Let $(X_i)_{i \in \Z}$ be a sequence of smooth varieties and $\r{A}$ be a symmetric sheaf $\d{Z}$-algebra on $(X_i)_{i\in \Z}$. Then there is an autoequivalence $\alpha$ on $\Gr(\r{A})$ inducing a commutative diagram for each $m$
\begin{displaymath}
\xymatrix{
\Gr(\r{A})\ar[rr]^{(-)_m} \ar[d]_{\alpha} && \Qcoh(X_m)\ar[d]^{\tr \omega_{X_m}}\\
\Gr(\r{A})\ar[rr]_{(-)_{m+2}} && \Qcoh(X_m)
}
\end{displaymath}
\end{proposition}
\begin{proof}
By Theorem \ref{thm:twisting}, $\r{A}$ is Morita equivalent to a symmetric sheaf $\d{Z}$-algebra $\d{S}(\r{E})$  in standard form with $\r{E} \in \bimod(X-Y)$ (using the notation \ref{conv:Z_n}). Moreover by  \cite[Lemma 3.1.7.]{VdB_12}, we have
\[ \r{E}^{*2}\cong\omega_{X}^{-1}\tr \r{E} \tr \omega_{Y} \]
Hence the twist by  the sequence of line bundles $({\omega_{X_i}})_{i \in \d{Z}}$ yields an equivalence 
\[ \r{T}: \Gr(\d{S}(\r{E}))\stackrel{\cong}{\mor} \Gr(\omega^{-1}\tr \d{S}(\r{E}) \tr \omega)\stackrel{\cong}{\mor} \Gr(\d{S}(\r{E}^{*2})) \]
where we used the short-hand notation
\[ \left(\omega^{-1}\tr \d{S}(\r{E}) \tr \omega \right)_{m,n} = \omega_{X_m}^{-1} \tr \d{S}(\r{E})_{m,n} \tr \omega_{X_n} \]
Next, the construction of a standard symmetric sheaf $\d{Z}$-algebra implies that there is an equivalence $\Psi:\Gr(\d{S}(\r{E})(2))\mor \Gr(\d{S}(\r{E}^{*2}))$ (where we used the notation (\ref{eq:twistingsheafzalgebra})). We now simply define
\[ \alpha := (-2)\circ\Psi^{-1}\circ \r{T}: \Gr(\d{S}(\r{E}))\mor \Gr(\d{S}(\r{E}^{*2}))\mor \Gr(\d{S}(\r{E})(2))\mor \Gr(\d{S}(\r{E})) \]
\end{proof}
In the commutative case (discussed in section \ref{subsec:rank(2,2)case}), the $0^{\textrm{th}}$ truncation functor coincides with the pushforward functor in the following sense:
\begin{theorem} \label{trm:commutativepushforward}
Let $\r{V}$ be a locally free sheaf of rank $2$ on $X$ and consider the associated symmetric sheaf $\d{Z}$-algebra $\d{S}({}_{\Id}(\r{V})_{\Id})$. Let 
\[
\Phi: \Proj(\d{S}({}_{\Id}(\r{V})_{\Id}))\mor \Qcoh(\d{P}_X(\r{V}))
\]
 be the equivalence provided by Corollary \ref{cor:commutativealgebra}. Then the following diagram commutes
\begin{displaymath}
\xymatrix{
& \Gr(\d{S}({}_{\Id}(\r{V})_{\Id}))\ar[dr]^{(-)_0}\\
\Proj(\d{S}({}_{\Id}(\r{V})_{\Id}))\ar[ur]^{\omega}\ar[dr]_\Phi&&\Qcoh(X)\\
& \Qcoh(\d{P}_X(\r{V}))\ar[ur]_{\pi_*}
}
\end{displaymath}
\end{theorem}
\begin{proof}
Let $Z:=\d{P}_X(\r{V})$ and $\r{A} := \d{S}({}_{\Id}(\r{V})_{\Id})$. The  formula we need to prove explicitly is
\[ 
\pi_*\Big( \widetilde{\ds_i (-)\tr \r{T}_i} \Big) \cong \big(\omega (-)\big)_0 
\]
where $\T_i=\big(\big(\bigwedge^2\r{V}\big)^{\floor{\frac{i}{2}}}\big)_{i\in \d{Z}}$ is given as in the statement of Lemma \ref{lem:commutativealgebra}.\\
Now by Lemma \ref{lem:adjunctiontruncation} and the definition of $\omega$, the functor $\big(\omega(-)\big)_0$ is right adjoint to $p\big((-)\tr e_0\r{A}\big)$. Another formal computation using Corollary \ref{cor:commutativealgebra} shows that $\pi_*\Big(\widetilde{\ds_i (-)\tr \r{T}_i}\Big)$ is right adjoint to the functor $\r{T}^{-1} \left( (p \circ \Gamma_*) \big( \pi^*(-) \big) \right)$. This functor in turn being equal to $p\left( \big(\pi_*\big(\pi^*(-)(i)\big)\tr \r{T}_i^{-1} \big)_i \right)$, which by the projection formula, simplifies to $p\left(\big((-)\tr \pi_*\r{O}_Z(i)\tr \r{T}^{-1}_i\big)_i \right)$. The unicity of adjoint functors thus reduces the claim to proving the isomorphism
\begin{equation}
\label{eq:2periodicity} \big((-)\tr \pi_*\r{O}_Z(i)\tr \r{T}_i\big)_i \cong (-)\tr e_0\r{A}
\end{equation}
Since $\rk(\r{E})\ge 2$, \cite[Proposition II.7.11.a]{Hartshorne77} implies that there is an isomorphism $\pi_*\big(\r{O}_Z(i)\big)=\Sym_X(\r{V})_i$. Now, by the choice of $\r{T}_i$, we have $\Sym_X(\r{V})_i=\r{A}_{0i}\tr \r{T}_i$. (\ref{eq:2periodicity}) thus becomes
\[ \big((-)\tr \pi_*\r{O}_Z(i)\tr \r{T}^{-1}_i\big)_i = \big((-)\tr  \r{A}_{0i}\tr \r{T}_i\tr \r{T}^{-1}_i\big)_i = \big( (-) \tr \r{A}_{0i} \big)_i = (-)\tr e_0\r{A} \]
proving the claim.
\end{proof}

We also have 1-periodicity for the truncation functors in this case:

\begin{proposition}\label{prop:1-periodicity}
Let $\r{V}$ be a locally free sheaf of rank 2 on $X$ and $\d{S}({}_{\Id} \r{V}_{\Id})$ the associated symmetric sheaf $\d{Z}$-algebra. Then there is an equivalence $\beta$ and for each $n$, a line bundle $\r{L}_n$ on $X$ making the diagram
\begin{displaymath}
\xymatrix{
\Gr(\d{S}({}_{\Id} \r{V}_{\Id}))\ar[rr]^{(-)_n} \ar[d]_{\beta} && \Qcoh(X)\ar[d]^{- \tr \r{L}_n}\\
\Gr(\d{S}({}_{\Id} \r{V}_{\Id}))\ar[rr]_{(-)_{n+1}} && \Qcoh(X)
}
\end{displaymath}
commute.
\end{proposition}
\begin{proof}
By Lemma \ref{lem:commutativealgebra}, there is a sequence of $X-X$-bimodules $\r{T}_i$ such that the following is an equivalence of categories
\[ \Gr(\d{S}({}_{\Id} \r{V}_{\Id})) \mor \Gr(\Sym_X(\r{V})): \left( \r{M}_i \right)_i \mapsto \bigoplus_i \r{M}_i \tr \r{T}_i \]

Let $(-1)$ denote the inverse shift functor on $\Gr(\Sym_X(\r{V}))$, i.e. $(\r{M}(-1))_i = \r{M}_{i-1}$ and define $\beta$ as the autoequivalence making the diagram
\begin{displaymath}
\xymatrix{
\Gr(\d{S}({}_{\Id} \r{V}_{\Id}))\ar[rr]^{\r{T}}\ar[d]_\beta&&\Sym_X(\r{V})\ar[d]^{(-1)}\\
\Gr(\d{S}({}_{\Id} \r{V}_{\Id}))\ar[rr]_{\r{T}} && \Sym_X(\r{V})
}
\end{displaymath}
commute. Since we clearly have $(-)_{n+1} \circ (-1)=(-)_n$, we get the required result by choosing the line bundle $\r{L}_n:=\r{T}_n\tr \r{T}^{-1}_{n+1}$ with $\r{T}_n$ as in the proof of Lemma \ref{lem:commutativealgebra}.
\end{proof}

\begin{remark}
the previous result of  1-periodicity clearly implies 2-periodicity after repeated application in the sense that
$$(-)_{n+2}\circ \beta^2=\big(\r{L}_{n+1}\tr\r{L}_n\big)\tr (-)_n$$
hence one can wonder whether this periodicity coincides with proposition \ref{prop:2-periodicity}. This is not the case in general. Indeed, from the explicit form of $\r{T}$ in proposition \ref{prop:2-periodicity} and $\beta$ in Proposition $\ref{prop:1-periodicity}$, we obtain 
$$\r{L}_n=\big(\bigwedge^{2}\r{V}\big)^{\floor{\frac{n}{2}}}\tr \big(\bigwedge^{2}\r{V}\big)^{-\floor{\frac{n+1}{2}}}$$ and $\r{L}_{n+1}\tr \r{L}_n=\left( \bigwedge^2(\r{V}) \right) ^{-1}$, which obviously does not coincide with $\omega_{X/S}$  in general.
\end{remark}

\section{Noetherianity of $\Gr(\d{S}(\r{E}))$}

As explained in the introduction, it is the case of a bimodule $\r{E}$ of rank $(4,1)$ that we are particularly interested in. This section is dedicated to proving one of the important geometric properties of $\d{S}(\r{E})$ in this setting:

\begin{theorem} \label{thm:noeth}
Let $X$ and $Y$ be smooth varieties over $\k$ and $\r{E} \in \bimod(X-Y)$ be locally free of rank (4,1). Then the category $\Gr(\d{S}(\r{E}))$ is locally noetherian.
\end{theorem}

\begin{convention}
\label{conv:conditions}
Throughout this section we will always assume that $X,Y$ and $\r{E} \in \bimod(X-Y)$ satisfy the conditions in Theorem \ref{thm:noeth}.
\end{convention}
 The next lemma shows that under these assumptions, the bimodule $\r{E}$ can written in a convenient form using a line bundle on $Y$ and a finite map $f$ of degree 4. For future reference, we state this lemma in a slightly more general setting
 
\begin{lemma}
\label{lem:inducedbyline}
 Assume that $X,Y$ are schemes of finite type over $\k$ and $\r{E}$ is a locally free $X-Y$-bimodule of rank $(n,1)$.
Then there is a line bundle $\r{L}$ on $Y$ and a finite surjective morphism\footnote{note that $f$ is automatically flat here} $f:Y\mor X$ of degree $n$ such that 
$\r{E}\cong {}_f\r{L}_{\Id}$ (see Definition \ref{def:standardformbimodule} ).
\end{lemma}
\begin{proof} Let $W\subset X\times Y$ be the scheme theoretic support of $\r{E}$ and denote the projections
$W\mor X$, $W\mor Y$ by $g,h$ respectively:

\begin{center}
\begin{tikzpicture}
\matrix(m)[matrix of math nodes,
row sep=2em, column sep=2em,
text height=1.5ex, text depth=0.25ex]
{  & \Supp(\r{E})=W & \\
 & X \times Y & \\
X & & Y \\};
\path[right hook->,font=\scriptsize]
(m-1-2) edge node[right]{$\iota$} (m-2-2);
\path[->,font=\scriptsize]
(m-1-2) edge [bend right=30] node[left]{$g$}(m-3-1)
        edge [bend left=30]node[right]{$h$}(m-3-3)
(m-2-2) edge node[above]{$\pi_X$}(m-3-1)
        edge node[above]{$\pi_Y$}(m-3-3);
\end{tikzpicture}
\end{center}

 By definition $g,h$ are finite morphisms. Furthermore $\r{E} \cong {}_g\r{F}_h$ for $\r{F}\in \coh(W)$ such that $\Supp \r{F}=W$.
By Lemma \ref{triv} below we conclude that $h$ is an isomorphism and that $\r{F}$ is a line bundle on $Z$. Put
$\r{L}=h_\ast\r{F}$, $f=gh^{-1}$. Then $\r{E}\cong {}_f\r{L}_{\Id}$.  Since $\r{L}$ is a line bundle, $f_\ast\r{L}$ and
 $f_\ast \r{O}_Y$ are locally isomorphic (e.g. by Lemma 5.3.8 below). So $f_\ast\r{O}_Y$ is locally free of rank $n$ as well and therefore $f$ is flat of degree $n$.
\end{proof}
\begin{lemma} 
\label{triv} Assume that $h:W\mor Y$ is a finite morphism between $\k$-schemes of finite type, $\r{F}$ is a coherent sheaf on $W$ whose scheme theoretic support is $W$ and $h_\ast \r{F}$ is locally free of rank one. Then $h$ is an isomorphism and $\r{F}$ is a line bundle on $Z$.
\end{lemma}
\begin{proof} Since $h$ is finite it is affine, we may assume that $Y=\Spec R$, $W=\Spec S$ and $\r{F}=\tilde{F}$ for $F$ an $S$-module which is invertible as $R$-module. The composition of
\[
R\xrightarrow{h} S\xrightarrow{s\mapsto (f\mapsto sf)} \End_R(F)\cong R
\]
is the identity and the middle map is injective since $W$ is the scheme-theoretic support of $\r{F}$. It follows that all maps
are isomorphisms. The claim follows. 
\end{proof}

\begin{convention}
\label{conv:standardform}
Following the above lemma, we shall assume that  $\r{E}$ is given in the above form, i.e. $\r{E} = {}_f ( \r{L} ) _{\Id}$ for some finite flat morphism $f:Y\mor X$ of degree $4$.
\end{convention}

\subsection{Restricting to an Open Subset}
The first step in the proof of Theorem \ref{thm:noeth} is to show that there is an appropriate notion of restricting a sheaf $\Z$-algebra  to an open subset and that the statement of Theorem \ref{thm:noeth} can be reduced to an open cover in this sense.\\[\medskipamount]
To this end, we let $\r{A}$ denote a sheaf $\Z$-algebra over a sequence of smooth varieties $(X_i)_{i \in \Z}$ and  $\r{U}=(U^i)_{i\in \Z} $ be a sequence of affine open subsets $U^i\subset X_i$. For an $X_m-X_{m+1}$- bimodule $\r{F}$, and a graded $\r{A}$-module $\r{M}$ we will use the notation $|_\r{U}$ to denote the restriction to the corresponding open subset. I.e.
\begin{eqnarray*}
\left.\r{F}\right|_\r{U} & := & \left.\r{F}\right|_{U^m \times U^{m+1}} \\
\left( \left.\r{A}\right|_\r{U} \right)_{m,n} & := &  \left.\left(\r{A}_{m,n}\right)\right|_\r{U} \ =  \left.(\r{A}_{m,n})\right|_{U^m \times U^{n}} \\
\left( \left.\r{M}\right|_\r{U} \right)_{m} & := &  \left.\left(\r{M}_{m}\right)\right|_{U^m}
\end{eqnarray*}
To ensure that the restrictions of $\r{A}$ to an open subset remains a sheaf $\d{Z}$-algebra, we need the following technical condition:

\begin{lemma}\label{lem:affinereduction}
Let $\r{A}$ be a sheaf $\d{Z}$-algebra and  $\r{U}$ as above such that for $m,n \in \Z$: 
\[
\Supp(\left.(\r{A}_{m,n})\right|_{U^m \times X_n}) \subset U^m \times U^n \text{ and } \Supp(\left.(\r{A}_{m,n})\right|_{X_m \times U^n}) \subset U^m \times U^n
\] 
then
\begin{enumerate}[i)]
\item $\left. \r{A} \right|_\r{U}$ has an induced algebra structure.
\item Restriction of modules to $\r{U}$ defines a functor $|_\r{U}: \Gr(\r{A}) \rightarrow \Gr(\left. \r{A} \right|_\r{U})$
\end{enumerate}
\end{lemma}

\begin{proof}
\begin{enumerate}[i)]
\item We must show that for all $l,m,n \in \d{Z}$ there are multiplication morphisms $\left. \r{A}_{l,m} \right|_\r{U} \tr \left.\r{A}_{m,n}\right|_\r{U} \rightarrow \left.\r{A}_{l,n}\right|_\r{U}$ induced by the morphisms ${\r{A}_{l,m} \tr \r{A}_{m,n} \rightarrow \r{A}_{l,n}}$.\\
It is evident that the latter induces a morphism of $U^l-U^n$-bimodules as follows:
\[ \left. \left( \r{A}_{l,m} \tr \r{A}_{m,n} \right) \right|_\r{U} \rightarrow \left. \r{A}_{l,n} \right|_\r{U} \]
Now the claim follows from the following chain of isomorphisms:
\begin{eqnarray*}
\left. \left( \r{A}_{l,m} \tr \r{A}_{m,n} \right) \right|_\r{U} & = & \left. \left( \pi_{X_l,X_n *} \left( \pi_{X_l,X_m}^*(\r{A}_{l,m}) \tr_{X_l \times X_m \times X_n} \pi_{X_m,X_n}^*(\r{A}_{m,n}) \right) \right) \right|_{U^l \times U^n} \\
& = & \pi_{U^l,U^n *} \left( \left. \left( \pi_{X_l,X_m}^*(\r{A}_{l,m}) \tr_{X_l \times X_m \times X_n} \pi_{X_m,X_n}^*(\r{A}_{m,n}) \right) \right|_{U^l \times X_m \times U^n}\right) \\
& = & \pi_{U^l,U^n *} \left( \left. \pi_{X_l,X_m}^*(\r{A}_{l,m}) \right|_{U^l \times X_m \times U^n} \tr \left. \pi_{X_m,X_n}^*(\r{A}_{m,n}) \right|_{U^l \times X_m \times U^n}\right) \\
& = & \pi_{U^l,U^n *} \left( \pi_{U^l,X_m}^*(\left. \r{A}_{l,m}\right|_{U^l \times X_m})  \tr_{U^l \times X_m \times U^n} \pi_{X_m,U^n}^*(\left. \r{A}_{m,n}\right|_{X_m \times U^n}) \right) \\
& = & \pi_{U^l,U^n *} \left( \pi_{U^l,U^m}^*(\left. \r{A}_{l,m}\right|_{U^l \times U^m})  \tr_{U^l \times U^m \times U^n} \pi_{U^m,U^n}^*(\left. \r{A}_{m,n}\right|_{U^m \times U^n}) \right) \\
& = & \left. \r{A}_{l,m} \right|_\r{U} \tr \left. \r{A}_{m,n} \right|_\r{U}
\end{eqnarray*}
where $\pi_{U^l,X_m}$ and $\pi_{U^l,U^m}$ are the projections ${\pi_{U^l,X_m}: U^l \times X_m \times U^n \rightarrow U^l \times X_m}$ and ${\pi_{U^l,U^m}: U^l \times U^m \times U^n \rightarrow U^l \times U^m}$, with similar definitions for $\pi_{X_m,U^n}$ and $\pi_{U^m,U^n}$.\\
The first equality is the definition of tensor product of bimodules \[\bimod(X_l-X_m) \times \bimod(X_m-X_n) \rightarrow \bimod(X_l-X_n) \]
The second equality follows from the commutation of pushforward and restriction of sheaves.
The third equality follows from the commutation of tensor product of sheaves and restriction.
The fourth equality follows from the commutation of pullback and restriction of sheaves.
The fifth equality follows the assumption of the lemma.
The last equality is the definition of multiplication 
\[ \bimod(U^l-U^m) \times \bimod(U^m-U^n) \rightarrow \bimod(U^l-U^n) \]
\item This essentially reduces to showing $\left. \left( \r{M}_i \tr \r{A}_{i,j} \right) \right|_{U_j} = \left( \left. \r{M}\right|_\r{U} \right)_i \tr \left( \left. \r{A}\right|_\r{U} \right)_{i,j} $ which is completely similar to i).
\end{enumerate}
\end{proof}

Our main motivation to study restriction of sheaf $\Z$-algebra lies in the following result whose proof is straightforward:

\begin{lemma} \label{lem:opencover}
Let $\r{U}_\alpha$ be a finite set of sequences such that for each $i \in \Z$, $\bigcup_\alpha (U^i)_\alpha=X_i$. Assume that $\r{A}$ is a sheaf $\d{Z}$-algebra such that the conditions in Lemma \ref{lem:affinereduction} are satisfied for all $\r{U}_\alpha$, then
\[  \forall \alpha: \left. \r{M} \right|_{\r{U}_\alpha} \in \Gr(\left. \r{A} \right|_{\r{U}_\alpha}) \textrm{ is noetherian } \Rightarrow \r{M} \in \Gr(\r{A}) \textrm{ is noetherian } \]
\end{lemma}
\begin{proof}
Suppose we are given an ascending chain of sub-objects of $\r{M}^n \subset \r{M}$ in $\Gr(\r{A})$ such that the restriction of this chain to all of the sequence $\r{U}_\alpha$ stabilizes. As there are only finitely many $\r{U}_\alpha$, there is an $N \in \d{N}$ such that for  all $n \geq N$ and  for all $\alpha$: $\left. (\r{M}^n)\right|_{\r{U}_\alpha} = \left. (\r{M}^{n+1})\right|_{\r{U}_\alpha}$. The graded modules $\r{M}^n$ and $\r{M}^{n+1}$ must coincide.
\end{proof}
Following the convention \ref{conv:standardform}, we now consider the case where $\r{A}=\d{S}(\r{E})$ for $\r{E}=_{f}{\r{L}}_{\Id}$ where $f:Y\mor X$ is finite of degree 4. Then for an affine open subset $U\subset X$ we define the associated sequence $\r{U}$ by $U^i \subset X_i$ as follows:
\[ U^i = \begin{cases} U & \textrm{ if $i$ is even}\\ f^{-1}(U) & \textrm{ if $i$ is odd} \end{cases} \]
Note that $U^i$ is indeed an affine open subset because $f$ is a finite morphism. The results of Lemma \ref{lem:affinereduction} in this context can be stated as follows:
\begin{corollary} \label{cor:openrestriction}
For any $U \subset X$,\begin{enumerate}[i)]
\item $\left. \d{S}(\r{E}) \right|_U$ has an algebra structure induced by $\d{S}(\r{E})$
\item There is a functor $|_U: \Gr(\d{S}(\r{E}) ) \rightarrow \Gr(\left. \d{S}(\r{E}) \right|_U )$ 
\item There is an isomorphism of symmetric sheaf $\d{Z}$-algebras: $\left. \d{S}(\r{E}) \right|_U \cong \d{S}(\left. \r{E} \right|_U) $
\end{enumerate}
\end{corollary}

\begin{proof}
\emph{i+ii)} As $\r{E}$ is given as ${}_f (\r{L})_{\Id}$ following convention \ref{conv:standardform}, the conditions in lemma \ref{lem:affinereduction} are trivially satisfied for $\r{A} = \d{S}(\r{E})$. For iii) We first show that for all $m \in \d{N}$ there is a natural isomorphism
	\begin{equation} \label{eq:dualrestriction} \theta_\r{E}:\left.(\r{E}^{*m})\right|_U = \left( \left.\r{E}\right|_U\right)^{*m} \end{equation}
Using remark \ref{rem:dual} we see by induction that for each $m \geq 0$ there is a line bundle $\r{L}_m$ such that
\begin{eqnarray*} \notag \r{E}^{*2m} &=& {}_f (\r{L}_m)_{\Id} \\
\label{eq:forappendixc} \r{E}^{*2m+1} &=&  {}_{\Id} (\r{L}_m)_f
\end{eqnarray*} where $\r{L}_0=\r{L}$. The explicit form of the dual in remark \ref{rem:dual} shows that it suffices to exhibit isomorphisms
\[ \left. {}_f \left( \ShHom_Y(\r{L}_m, f^! \r{O}_X)\right)_{\Id} \right|_U \cong {}_{f|_U} \left( \ShHom_{f^{-1}(U)}(\left.\left(\r{L}_m\right)\right|_{f^{-1}(U)}, (f|_U)^! \r{O}_U)\right)_{Id_U} \]
However as restriction to open affine subsets commutes with $f_*$, $\ShHom_Y$ and $f^!$, this isomorphism is immediate. The case $m<0$ follows easily by (\ref{eq:yoneda}).\\%	This implies that (\ref{eq:dualrestriction}) is valid for $m<0$ as well: indeed, it suffices to show this for $m=-1$. In this case. there is at least a morphism $\nu_{\r{E}}:^{*}({\r{E}}\arrowvert_U) \mor \left(^{*}{\r{E}}\right)\arrowvert_U$. To show that $\theta_\r{E}$ is an isomorphism, we may apply $(-)^*$, since this is a fully faithful functor. This yields a commutative diagram
%\begin{displaymath}
%\xymatrix{
%\big({^{*}{(\r{E}\arrowvert_U})}\big)^* \ar[r]^\simeq \ar[d]_{\nu_\r{E}^*} & \r{E}\arrowvert_U\ar[dd]^{Id}\\
%\big( ({^{*}{\r{E}}})^*\arrowvert_U   \big)\ar[d]_{\theta_\r{E}} \\
%\big( ({^{*}{\r{E}}})^*\big) \arrowvert_U \ar[r]_\simeq & \r{E}\arrowvert_U}	
%\end{displaymath}
%proving the claim.\\[\medskipamount]
Finally, the naturality of $\theta_\r{E}$ immediately implies that the restricted unit morphisms $i_m\arrowvert_U$ coincides with 
\begin{equation*}  \label{eq:localunitmorphism} {}_{\Id}\left(\r{O}_{U^m} \right)_{\Id} \mor  \left( \left.\r{E}\right|_U \right)^{*m} \tr \left( \left.\r{E}\right|_U \right)^{*m+1} 
\end{equation*}
	Implying in particular that $\theta_\r{E}$ induces an isomorphism $$i_m( {}_{\Id}\left(\r{O}_{U^m} \right)_{\Id}) \cong i_m( {}_{\Id}\left(\r{O}_{X^m} \right)_{\Id})\arrowvert_{U^m}$$ and we can extend $\theta_\r{E}$ to an isomorphism 
\[\left. \theta:\d{S}(\r{E}) \right|_U \cong \d{S}(\left. \r{E} \right|_U) \qedhere \]
\end{proof}

\subsection{Covering by Relative Frobenius Pairs}
\label{subsec:covering}
Lemma \ref{lem:opencover} shows that proving that an $\d{S}(\r{E})$-module is noetherian can be done over an affine open cover. In this subsection we construct an open cover $X = \bigcup_l U_l$ for which the categories $\Gr(\left. \d{S}(\r{E}) \right|_{U_l})$ can explicitly described (see Lemma \ref{thm:coverconditions}). Over this cover, the rings of sections satisfy a relative version of the Frobenius property as introduced in the paper \cite{Genpreprojective}. We begin by recalling the basic definition and results of \cite{Genpreprojective} for the benefit of the reader.
\begin{definition} \label{def:generalizedpreprojective}
We say that a morphism of rings $R\mor S$ is \emph{relative Frobenius} of rank $n$ if:
\begin{itemize}
\item $S$ is a free $R$-module of rank $n$.
\item $\Hom_R(S,R)$ is isomorphic to $S$ as $S$-module.
\end{itemize}
\end{definition}
\begin{remark}\label{rem:relfrob}
It is clear  that if $R$ is a field, the condition of $S/R$ being relative Frobenius coincides with $S$ being a finite dimensional Frobenius algebra in the classical sense.
\end{remark}

We shall need the following notation: for a relatively Frobenius pair, let $M:={}_R S_S$. This $R-S$-bimodule can be considered a $R\ds S$ bimodule by letting only the $R$-component act on the left and only the $S$-component on the right. Similarly, we let $N:={}_S S_R$ and consider it an $R\ds S$-bimodule by letting only the component $S$ act on the left and only the component $R$ act on the right. We now define
$$T(R,S):=T_{R\ds S}(M\ds N)$$
Note that by construction, in degree $2$, we have $M\tr_{R\ds S} M=N \tr_{R\ds S} N=0$, hence
$$T(R,S)_2=\left( M_{R \ds S} N \right) \ds \left( N \tr_{R\ds S} M\right)=\left({}_R S\tr_S  S_R\right) \ds \left({}_S S\tr_R S_R\right)$$
The algebra we will be concerned in will be a quotient of $T(R,S)$ as follows: let $\lambda$ be a generator
of $\Hom_R(S,R)$ as an  $S$-module. The $R$-bilinear form $\bl{a}{b}:=\lambda(ab)$ is clearly nondegenerate and hence we can find dual $R$-bases $(e_i)_i$, $(f_j)_j$  satisfying
\[
\lambda(e_if_j)=\delta_{ij}
\]
\begin{definition} \label{def}
For a relative Frobenius pair $S/R$, the \emph{generalized preprojective algebra} $\Pi_R(S)$ is given by
\[
T(R,S)/(\text{rels})
\]
where the relations are in degree 2 given by
\[
1\otimes 1\in {}_R S\otimes_S S_R
\]
\[
\sum_i e_i\otimes f_i\in {}_S S\otimes_R S_S
\]
\end{definition}

\begin{remark}
If $S$ is the ring $R^{\ds n}$. Then $\Pi_R(S)$ is isomorphic to the preprojective algebra over $R$ associated to the quiver with one central vertex and $n$ outgoing arrows. (See \cite[Lemma 1.5]{Genpreprojective})
\end{remark}

We shall use the following result from \cite{Genpreprojective}:
\begin{theorem} \label{thm:pirsnoetherian}
Let $S/R$ be relative Frobenius of rank 4 and assume $R$ is noetherian, then $\Pi_R(S)$ is a noetherian algebra.
\end{theorem}

Throughout, we shall make use of the following lemma, well-known to experts:

\begin{lemma} \label{lem:Xlocallytrivial}
Let $f:Y\mor X$ be a finite morphism of smooth varieties.\\
Let $\r{L}$ be a line bundle on $Y$ and $p \in X$. Then there is an open subset $U \subset X$ containing $p$, such that $\left.\r{L}\right|_{f^{-1}(U)} \cong \r{O}_{f^{-1}(U)}$.
\end{lemma}
\begin{proof}

Since affine open subsets form a base for the topology on $X$ and $f$ is affine (as it is finite), we can reduce to the case where ${X = \Spec(R)}$, ${Y = \Spec(S)}$ are affine varieties over $\k$ and $S$ is finitely generated over $R$ and $\r{L} = \tilde{L}$ for some invertible $S$-module $L$. Let $\f{p}$ be the prime ideal in $\Spec(R)$ corresponding to $f(p) \in X$, then $S_\f{p} := S \tr_R R_{\f{p}}$ is a semi-local ring, hence every finitely generated projective $S_\f{p}$-module of constant rank is free and in particular the Picard group of $S_\f{p}$ is trivial. Consequently, there exists an $l \in L$ such that
\[ S_{\f{p}} \stackrel{ \cdot l }{\mor} L_{\f{p}} \]
is an isomorphism.\\
Now consider the morphism $S \stackrel{ \cdot l }{\mor} L$ with kernel $K$ and cokernel $C$. Then there is an exact sequence
\begin{equation} \label{eq:lem:locallytrivial} 0 \mor K \mor S \stackrel{ \cdot l }{\mor} L \mor C \mor 0 \end{equation}
$K$ is a finitely generated  $R$-submodule of $S$ by the noetherianity of $R$. $L$ is finitely generated over $R$, being an invertible $S$-module. It follows that $C$ is finitely generated over $R$ as a quotient of $L$. Now let $\alpha_1, \ldots, \alpha_n$ be a set of generators for $\k$, then as $K \tr R_\f{p} = 0$ there exist elements ${x_1, \ldots, x_n \in R \backslash \f{p}}$ such that ${\alpha_1 x_1 = \ldots = \alpha_n x_n = 0}$. Set ${x:= x_1 \cdot \ldots \cdot x_n \in R \backslash \f{p}}$, then $\alpha \cdot x = 0$ for all $\alpha \in K$. Similarly there is a $x' \in R \backslash \f{p}$ such that $\beta \cdot x' = 0$ for all $\beta \in C$. Now define $z = x \cdot x'$, then ${K \tr R_z = C \tr R_z =0}$ implying that $\cdot l$ defines an isomorphism 
\[ S \tr R_z \stackrel{\cong}{\mor} L \tr R_z\]
$U = \Spec(R_z)$ then is the desired open subset.
\end{proof}

We can now prove the main lemma of this subsection, which yields a cover over which many useful geometric properties are satisfied:

\begin{lemma} \label{thm:coverconditions}
Write $\r{E} = {}_f ( \r{L} ) _{\Id}$ as in Lemma \ref{lem:inducedbyline} . There is a finite cover $X = \bigcup_l U_l$ by affine open subsets $U_l = \Spec(R_l)$ such that:
\begin{enumerate}[i)]
\item $\left.\r{L}\right|_{f^{-1}(U_l)}$ is a trivial $\r{O}_{f^{-1}(U_l)}$-module
\item $\left.\omega_Y\right|_{f^{-1}(U_l)}$ is a trivial $\r{O}_{f^{-1}(U_l)}$-module
\item $\left.\omega_X\right|_{U_l}$ is a trivial $\r{O}_{U_l}$-module
\item $f^{-1}(U_l) = \Spec(S_l)$ where $S_l/R_l$ is relative Frobenius of rank 4.
\end{enumerate} 
\end{lemma}
\begin{proof}
We first note the following two facts:
\begin{itemize}
\item Let $\Spec(R)$ be an affine open subset on which $i)$, $ii)$, $iii)$ or $iv)$ holds. Then the same statement holds for any standard open $\Spec(R_f) \subset \Spec(R)$. This is obvious for $i)$, $ii)$ and $iii)$. For $iv)$ it follows from \cite[Lemma 3.1]{Genpreprojective}.
\item Let $\Spec(R)$ and $\Spec(R')$ be affine open subsets of $X$, then their intersection is covered by  open subsets which are simultaneously distinguished in each space, in other words subsets of the form $\Spec(R_f) = \Spec(R'_g)$ 
\end{itemize}
By these two facts it suffices to find affine open covers for i), ii), iii) and iv) separately. For i) and ii) such a cover exists by lemma \ref{lem:Xlocallytrivial} and the fact that $\omega_Y$ is a line bundle on the smooth variety $Y$. The existence of a cover satisfying iii) is immediate from the fact that $\omega_X$ is a line bundle. We have reduced the claim to exhibiting a cover satisfying iv).\\
Now by Lemma \ref{lem:Xlocallytrivial}: $f^! \omega_X$ is completely determined by $f_* \left(f^! \omega_X \right)$ and we have an isomorphism of $f_* \r{O}_Y$-modules 
\begin{equation} \label{eq:omegashriek} f_* \left( f^! \omega_X \right) := \ShHom_X(f_* \r{O}_Y, \omega_X) \cong f_* \omega_Y \end{equation}
As moreover $f$ is also surjective and flat, there is a cover ${X = \bigcup_l U_l}$ with ${U_l = \Spec(R_l)}$ and ${f^{-1}(U_l) = \Spec(S_l)}$ where $S_l$ is a free $R_l$-module of rank 4 for each $l$. By the previous arguments we can assume that ii) and iii) are also satisfied on this cover. In this case, replacing $f$ by its restriction $f^{-1}(U_l) \mor U_l$, (\ref{eq:omegashriek}) reads
\[ f_* \left( f^! \r{O}_{U_l} \right) := \ShHom_{U_l}(f \r{O}_{f^{-1}(U_l)}, \r{O}_{U_l}) \cong f_* \r{O}_{f^{-1}(U_l)} \]
and taking sections yields the required isomorphism of $S_l$-modules:
\[ \Hom_{R_l}(S_l,R_l) \cong S_l  \qedhere \]
\end{proof}

\subsection{From Periodic $\d{Z}$-Algebras to Graded Algebras}
\label{subsec:Ztogradedalgebra}
The previous section showed how we can reduce the statement of Theorem \ref{thm:noeth} to the case where $X$ and $Y$ are affine, and satisfy some convenient geometric properties (see Lemma \ref{thm:coverconditions}). In this section, we provide a second technical tool which allows us to reduce to the case where the $\Z$-algebra comes from a graded algebra. The $\widehat{(-)}$-construction (see \ref{conv:hat}) assigns a (1-periodic) $\Z$-algebra to a graded algebra. In this section, we consider the converse problem. More precisely, we show that an $n$-periodic $\d{Z}$-algebras $A$ gives rise to a graded algebra $\overline{A}$ such that $\Gr(A)$ is a direct summand of the category $\Gr(\overline{A})$. We start by describing the following slight generalization of $\Z$-algebras in order to be able to easily apply the result in our required setting:
\begin{definition}
\label{def:bimoduleZalgebra}
Let $(R_i)_{i \in \Z}$ be a sequence of commutative rings. A \emph{bimodule $\Z$-algebra} over $(R_i)_{i \in \Z}$ is a collection of $R_i-R_j$-bimodules $A_{i,j}$ together with multiplication maps
$$A_{i,j}\tr_{R_j} A_{j,l} \mor A_{i,l}$$ and $R_i$-linear unit maps $R_i \mor A_{i,i}$ satisfying the usual $\d{Z}$-algebra axioms.
If $\forall i: \ R_i=R$, then $A$ is called a bimodule $\Z$-algebra over $R$.
\end{definition}

\begin{definition} \label{def:dperiodic}
Let $A$ be a $\d{Z}$-algebra over $(R_i)_{i \in \d{Z}}$ and $d>0$ an integer.\\
Assume that for each $i$, we have $R_{i+d}=R_i$. We say $A$ is $d$-periodic if there is an isomorphism of $\d{Z}$-algebras $\varphi: A \stackrel{\sim }{\mor} A(d)$. I.e. there is a collection of $R_i-R_j$-bimodule isomorphisms $\{ \varphi_{ij}: A_{i,j} \stackrel{\sim }{\mor} A_{i+d,j+d} \}_{i,j}$ compatible with the multiplication and unit maps.
\end{definition}

Let $A$ be $d$-periodic and let $\displaystyle R:=\bigoplus _{i=0}^{d-1} R_i$. We construct a graded $R$-algebra $\overline{A}$ as follows: let $\overline{A}_n$ be a $d\times d$-matrix with entries:

\begin{equation}
\label{eq:overline}
 \left( \overline{A}_n \right) _{i,j} = \left\{ \begin{array}{cl} A_{i, i+n} & \textrm{ if $j-i \equiv n \ (\mod \ d)$ } \\
0 & \textrm{ else} \end{array} \right. 
\end{equation}
(Where we use the convention that the numbering of rows and columns of the matrix starts at 0 instead of 1.)\\
By way of example,
\[ \overline{A}_1 = \left( \begin{array}{ccccc} 0 & A_{0,1} & 0 & \ldots & 0 \\
0 & 0 & A_{1,2} & \ldots & 0 \\
\vdots & \vdots  & \vdots & \ddots & \vdots \\
0 & 0 & 0 & \ldots & A_{d-2,d-1} \\
A_{d-1,d} & 0 & 0 & \ldots & 0 \end{array} \right) \]
Each $\overline{A}_n$ is naturally a left (resp. right) $R$-module by letting  a $d$-tuple $(r_0, \ldots r_{d-1})$ act as a diagonal matrix $D$ with entries $D_{ii}:=r_i$ on the left (resp. right).\\
Moreover, there is a canonical multiplication map
$$\overline{A}_n \tr_R \overline{A}_m \mor \overline{A}_{n+m}$$ given by the ordinary matrix multiplication and applying the periodicity isomorphisms $\phi_{ij}$ whenever necessary.
The $(R_i)_{i \in \d{Z}}$-linearity of the $\d{Z}$-algebra multiplication implies that the above maps are indeed $R$-bilinear.
\begin{lemma}\label{lem:periodictograded}
Suppose $A$ is $d$-periodic, then the above maps define a graded (unital) $R$-algebra structure on the $R$-module $\overline{A}:=\ds_{i \in \d{Z}} \overline{A}_i$
\end{lemma}

\begin{proof}
The reader checks that the compatibility of the periodicity isomorphisms with the $\d{Z}$-algebra multiplication maps implies that the multiplication is associative. The algebra has a unit given by 
\[ 1 = \left( \begin{array}{cccc} e_0 & 0 & \ldots & 0 \\
0 & e_1 & \ldots & 0 \\
\vdots & \vdots & \ddots & \vdots \\
0 & 0 & \ldots & e_{d-1} \end{array} \right) \in \overline{A}_0 \]
where $e_i$ is the unit in $A_{ii}$.
\end{proof}
There is a convenient description of the category of graded right $\overline{A}$-modules as follows: let ${M \in \Gr(\overline{A})}$. By definition we have a decomposition $M=\bigoplus_{i \in \d{Z}} M_i$. Moreover, each $R$-module $M_i$ in turn has a direct sum decomposition given by $M_i=\DS_{j=0}^{d-1} e_jM_i$. We define $M_i^j:=e_jM_i$. This decomposition allows us to give a description of the $\overline{A}$-module structure of $M$. For a matrix $\overline{a} \in \overline{A}_m$, $e_j.\overline{a}$ only has one nonzero entry at position $(j,j+m)$. It follows from the right $R$-structure on $A_m$ that $e_j\overline{a}=\overline{a}.e_{j+m}$ -where we consider $j+m$ mod $d$ following (\ref{eq:overline}). Thus the right action of $\overline{A}_m$ on $M_i^j$ becomes a map of the form $M_i^j\tr A_{j,j+m} \mor M_{i+m}^{j+m}$ or equivalently for $l=j+m$,
$$M^j_i \tr A_{j,l} \mor M_{i+l-j}^{l}$$

\begin{lemma} \label{lem:catC}
Suppose $A$ is $d$-periodic and let $\mathscr{C}$ be the category defined as follows:
\begin{itemize}
\item $\Ob(\mathscr{C})$ consists of  collection of $R$-modules $(M_i^j)_{i \in \d{Z}, 0 \le j \leq d-1}$, such that $M_i^j$ is an $R_j$-module together with multiplication maps 
$$\mu_{i,j,l}^M: M^j_i \tr A_{j,l} \mor M_{i+l-j}^{l}$$
for each $i,j,l$ (where $l$ and $i+l-j$ should be interpreted modulo $d$) satisfying the obvious compatibility condition for multiplication and unit.
\item a morphism $M\mor N$ in $\mathscr{C}$ is a collection $f_{i,j}$ of $R_j$- linear maps $M_i^j\mor N_i^j$ such that 
\[ f_{i+l-j,l}\circ \mu_{i,j,l}^M  = \mu_{i,j,l}^N \circ ( f_{i,j} \tr A_{j,l}) \]
\end{itemize}
Then there is a canonical isomorphism of categories $\mathscr{C}\cong \Gr(\overline{A})$
\end{lemma}

\begin{proof}
The above discussion shows that the assignment $M\mor (M_le_i)_{l \in\d{Z},0\le i\le n-1}$ is well defined and essentially surjective. A morphism of graded modules ${f:M \mor N}$ will satisfy $f(M_ie_j)\subset N_ie_j$ and we can define $f_{i,j}$ as the restriction to these submodules. The $A$-linearity guarantees that $(f_{i,j})_{i,j}$ indeed defines a morphism in $\r{C}$ and since $\ds M_ie_j=M$ it is clear that this assignment is faithful. Since any collection of maps $f_{i,j}$ satisfying the above compatibility with the multiplication will sum up to an $\overline{A}$-linear map, the assignment is also full.
\end{proof}

\begin{lemma} \label{lem:catC2}
There is a decomposition
\[ \mathscr{C} = \mathscr{C}_0 \oplus \ldots \oplus \mathscr{C}_{d-1} \]
where $\mathscr{C}_n$ is the full subcategory of $\mathscr{C}$ whose objects are collections of $R$-modules $(M_i^j)_{i \in \d{Z}, 0 \le j \leq d-1}$ where $M_i^j = 0$ unless $j - i \equiv n \ (mod \ d)$.
\end{lemma}
\begin{proof}
This follows immediately from the construction of $\mathscr{C}$ and the fact that $j-i = l- (l+i-j)$. Hence, if $( \r{M}_i^j )_{ij}$ is a non-zero object in $\mathscr{C}_n$, then so is $( \r{M}_{l+i-j}^l )_{ij}$ for all $l$.
\end{proof}

\begin{proposition} \label{prp:periodicalgebras}
There is an exact embedding of categories
\[ \overline{(-)} : \Gr(A) \hookrightarrow \Gr( \overline{A} ) \]
sucht that the essential image is a direct summand of $\Gr( \overline{A} )$.
\end{proposition}
\begin{proof}
Let $M$ be an $A$-module with multiplication maps $\mu_{i,m}:M_i \tr_R A_m \mor M_{i+m}$ and let $\r{C}$ be as above. We define an object $\overline{M}$ in $\mathscr{C}$ by 
\[ \overline{M}_i^j = \left\{ \begin{array}{cl} M_i & \textrm{ if $j \equiv i$ mod $d$ } \\
0 & \textrm{else} \end{array} \right. \]
where the multiplication is given by
\[ \overline{\mu}_{i,j,l}= \left\{ \begin{array}{cl} \mu_{i,l-j} & \textrm{ if $j \equiv i$ mod $d$ } \\
0 & \textrm{else} \end{array} \right. \]
This assignment clearly defines an exact embedding 
\[
\Gr(A) \stackrel{\simeq}{\mor} \mathscr{C}_0 \mono \mathscr{C}
\]
which finishes the proof by lemmas \ref{lem:catC} and \ref{lem:catC2}.
\end{proof}

\subsection{A Local Description of $\d{S}(\r{E})$}
In this final step in the preparation of the proof of theorem \ref{thm:noeth}, we complete  the local description of $\d{S}(\r{E})$. By \ref{thm:coverconditions}, we have reduced the claim to the case where $X$ and $Y$ are affine. By our hypothesis on $X$ and $\r{E}$ (see conventions \ref{conv:conditionsxy} and \ref{conv:standardform}), we assume that $X=\Spec(R)$ and $Y=\Spec(S)$ are affine varieties over $\k$ such that $S/R$ is relative Frobenius of rank 4 with induced morphism $f:Y\mor X$, that $\r{E}=_f(\r{L})_1$ for some line bundle $\r{L}$ on $Y$ and $\omega_X \cong \r{O}_X$, $\omega_Y \cong \r{L} \cong \r{O}_Y$.  After applying the global section functor, we obtain a bimodule $\Z$-algebra in the sense of \ref{def:bimoduleZalgebra} which is 2-periodic. The graded algebra associated to this $\Z$-algebra by the construction in section \ref{subsec:Ztogradedalgebra} is precisely the generalized preprojective algebra defined in Definition \ref{def:generalizedpreprojective} and studied in \cite{Genpreprojective}.

\medskip

 We start by introducing some auxiliary notations. Recall the convention \ref{conv:Z_n} and let $\r{A}$ be a sheaf $\d{Z}$-algebra over $X_i$. There is a $\d{Z}$-algebra over $\k$, $\Gamma(\r{A})$ defined in each component by
\[ \Gamma(\r{A})_{i,j} := \Gamma(X_i \times X_j, \r{A}_{i,j}) \]
since each component $\Gamma(\r{A})_{i,j}$ is an $R-S$, $R-R$, $S-S$ or $S-R$ bimodule depending on the parity of the indices, $\Gamma(\r{A})$ is in fact a $\d{Z}$-algebra over commutative groundring $R\ds S$ (compare with the discussion following Remark \ref{rem:relfrob}).
The  equivalence between quasi-coherent sheaves over an affine scheme and modules over the ring of global sections can easily be adapted to our setting to yield an equivalence:
\[ \Gamma: \Gr(\r{A}) \stackrel{\simeq}{\mor} \Gr(\Gamma(\r{A})): \{ \r{M}_n \}_{n\in \d{Z}} \mapsto \{ \Gamma(X_n,\r{M}_n) \}_{n\in \d{Z}} \]

The following is an immediate consequence of the assumptions of this section:
\begin{lemma}
The $\d{Z}$-algebra $\Gamma(\d{S}(\r{E}))$ is 2-periodic in the sense that 
\[ \Gamma(\d{S}(\r{E}))_{i,j}=\Gamma(\d{S}(\r{E}))_{i+2,j+2} \]
\end{lemma}
\begin{proof}
By \ref{prop:2-periodicity}, there are isomorphisms $\d{S}(\r{E}))_{i+2,j+2} \cong \omega_i^{-1} \tr \d{S}(\r{E})) \tr \omega_j$. By the assumptions in the beginning of this section, both canonical bundles are trivial, implying that $\d{S}(\r{E})_{i,j}=\d{S}(\r{E})_{i+2,j+2}$. The result follows after applying $\Gamma(-)$.
\end{proof}
Using Lemma \ref{lem:periodictograded}, the 2-periodic $\mathbb{Z}$-algebra $\Gamma(\d{S}(\r{E}))$ gives rise to a graded algebra $\overline{\Gamma( \d{S}(\r{E}))}$. We now prove that this algebra coincides with the construction outlined in Subsection\ref{subsec:covering} :

\begin{lemma} \label{lem:locallypirs}
Let $X = \Spec(R)$ and $Y = \Spec(S)$ be smooth affine varieties  such that $S/R$ is relative Frobenius of rank 4. Let $f: Y \rightarrow X$ be the induced morphism and $\r{E} = {}_f (\r{O}_Y)_{\Id}$. Then $\overline{\Gamma(\d{S}(\r{E}))} \cong \Pi_R(S)$.
\end{lemma}
\begin{proof}
Consider the quotient map 
\begin{equation*} %\label{eq:T-S-surjection-sheaves} 
\d{T}(\r{E}) \epi \d{S}(\r{E}) \end{equation*}
Taking global sections in each component $\Gamma(X_i\times X_j,(-)_{i,j})$ yields a surjection
\begin{equation*} %\label{eq:T-S-surjection-sheaves} 
\Gamma(\d{T}(\r{E}))\epi \Gamma(\d{S}(\r{E})). \end{equation*}
as $X_i \times X_j$ is affine.\\
Since the functor $\overline{(-)}$ preserves surjectivity (see Proposition \ref{prp:periodicalgebras}), we obtain a map
\begin{equation*} %\label{eq:T-S-surjection-sheaves} 
\pi:\overline{\Gamma(\d{T}(\r{E}))}\epi \overline{\Gamma(\d{S}(\r{E}))}. \end{equation*}
We first show that there is a canonical isomorphism  of $R\ds S$-modules 
\begin{equation}\label{eq:tensoralgebra}
\overline{\Gamma(\d{T}(\r{E}))}\cong T(R,S)
\end{equation}
 For this (as $\Gamma\left(\d{S}(\r{E})\right)$ is clearly generated in degrees $0$ and $1$) it suffices to show the following three facts
\begin{itemize}
\item $ \overline{\Gamma(\d{T}(\r{E}))}_0 \cong T(R,S)_0=R\ds S$ as rings
\item $\overline{\Gamma(\d{T}(\r{E}))}_1 \cong T(R,S)_1 \cong {}_R S_S\ds {}_S S_R$ as $R\ds S$-modules
\item  the multiplication map yields isomorphisms 
\[ \overline{\Gamma(\d{T}(\r{E}))}_1 \tr \overline{\Gamma(\d{T}(\r{E}))}_n \stackrel{\cong}{\mor} \overline{\Gamma(\d{T}(\r{E}))}_{n+1} \]
\end{itemize}
For the first statement, we compute:
\begin{eqnarray*}
\overline{\Gamma(\d{T}(\r{E}))}_0 & = & \left( \begin{array}{cc} \Gamma(\d{T}(\r{E}))_{0,0} & 0 \\ 0 & \Gamma(\d{T}(\r{E}))_{1,1} \end{array} \right) \\
& = & \left( \begin{array}{cc} \Gamma \left( X \times X, {}_{\Id} \left( \r{O}_X \right)_{\Id}\right) & 0 \\ 0 & \Gamma \left( Y \times Y, {}_{\Id} \left( \r{O}_Y \right)_{\Id}\right) \end{array} \right)
\end{eqnarray*}
moreover, we have
\begin{eqnarray*}
\Gamma \left( X \times X, {}_{\Id} \left( \r{O}_X \right)_{\Id}\right) & = & \Hom \left( \r{O}_{X \times X}, \Delta_* \left( \r{O}_X \right) \right) \\
& = & \Hom \left( \Delta^* \left( \r{O}_{X \times X} \right), \r{O}_X  \right) \\
& = & \Hom \left( \r{O}_X, \r{O}_X  \right) \\
& \cong & R 
\end{eqnarray*}
And similarly $\Gamma \left( Y \times Y, {}_{\Id} \left( \r{O}_Y \right)_{\Id}\right) \cong S$. combining these calculations yields 
\[ \overline{\Gamma(\d{T}(\r{E}))}_0   \cong   \left( \begin{array}{cc} R & 0 \\ 0 & S \end{array} \right) \cong R\ds S\]
In a completely similar fashion, we check the second condition:
\begin{eqnarray*}
\overline{\Gamma(\d{T}(\r{E}))}_1 & = & \left( \begin{array}{cc} 0 & \Gamma(\d{T}(\r{E}))_{0,1} \\ \Gamma(\d{T}(\r{E}))_{1,2} & 0 \end{array} \right) \\
& = & \left( \begin{array}{cc} 0 & \Gamma \left( X \times Y, \r{E} \right) \\ \Gamma \left( Y \times X, \r{E}^* \right) \end{array} \right) \\
& = & \left( \begin{array}{cc} 0 & \Gamma \left( X \times Y, {}_f(\r{O}_Y)_{\Id} \right) \\ \Gamma \left( Y \times X, {}_{\Id}(\r{O}_Y)_f \right) & 0 \end{array} \right) \\
& \cong & \left( \begin{array}{cc} 0 & {}_R S _S \\ {}_SS_R & 0 \end{array} \right) \cong {}_R S_S \ds {}_S S_R \\
\end{eqnarray*}
To check the final condition, we have the isomorphisms
\begin{equation*} \d{T}(\r{E})_{i,i+1} \tr \d{T}(\r{E})_{i+1,i+n+1} \mor \d{T}(\r{E})_{i,i+n+1} \end{equation*}
We  now apply the functor $\Gamma(X_i\times X_{i+n+1},-)$ and note that since all varieties are affine, the tensor product and $\Gamma(-)$ commute, resulting in an isomorphism

\begin{equation*} 
	\Gamma(\d{T}(\r{E}))_{i,i+1} \tr \Gamma(\d{T}(\r{E}))_{i+1,i+n+1} \mor\Gamma(\d{T}(\r{E}) )_{i,i+n+1}
\end{equation*} 
application of the functor $\overline{(-)}$ yields
\[ \overline{\Gamma(\d{T}(\r{E}))}_1 \tr \overline{\Gamma(\d{T}(\r{E}))}_n \stackrel{\simeq}{\mor} \overline{\Gamma(\d{T}(\r{E}))}_{n+1} \]
we have thus constructed the required isomorphism (\ref{eq:tensoralgebra}). Finally, we prove that the relations defining $\Pi_RS$ coincide with the kernel of $\pi$, i.e. there is a commutative diagram:
\begin{displaymath}
\xymatrix{
\overline{\Gamma(\d{T}(\r{E}))} \ar[d]_{\cong}\ar[rr]^\pi  & & \overline{\Gamma(\d{S}(\r{E}))}\ar[d]^\cong\\
T(R,S) \ar[rr]_{\overline{\pi}}& &\Pi_R(S)
}	
\end{displaymath}
The isomorphisms in the previous step yield isomorphisms:
\[ 
\zeta_0: \Hom_{X \times X}( {}_{\Id} \left(\r{O}_X \right)_{\Id}, \r{E} \tr \r{E}^*) \stackrel{\simeq}{\mor} \Hom_R(R,{}_RS_S \tr_S S_R) 
\]
\[ 
\zeta_1: \Hom_{Y \times Y}( {}_{\Id} \left(\r{O}_Y \right)_{\Id}, \r{E}^* \tr \r{E}) \stackrel{\simeq}{\mor} \Hom_S(S,{}_SS_R \tr_R S_S) 
\]
Recall that $\d{S}(\r{E})$ is defined as a quotient of $\d{T}(\r{E})$ by the relations given by the unit morphisms 
\[
\begin{cases}
i_0 \in \Hom_{X \times X}( {}_{\Id} \left(\r{O}_X \right)_{\Id}, \r{E} \tr \r{E}^*)\\
i_1 \in  \Hom_{Y \times Y}( {}_{\Id} \left(\r{O}_Y \right)_{\Id}, \r{E}^* \tr \r{E}) 
\end{cases}
\]
described in (\ref{eq:unitmorphism}).\\
Similarly $\Pi_R(S)$ is defined as a quotient of $T_R(S)$ by elements 
\[
\begin{cases}
\eta_0 \in \Hom_R(R,{}_RS_S \tr_SS_R)\\
\eta_1 \in \Hom_S(S,{}_SS_R \tr_RS_S).
\end{cases}
\]
Hence we must prove $\zeta_0(i_0)=\eta_0$ and $\zeta_1(i_1)=\eta_1$. To this end, note that there is a commutative diagram of isomorphisms
\begin{center}
\begin{tikzpicture}
\matrix(m)[matrix of math nodes,
row sep=2em, column sep=2em,
text height=1.5ex, text depth=0.25ex]
{ \Hom_{X \times Y}(\r{E}, \r{E}) & \Hom_{X \times X}({}_{\Id} \left( \r{O}_X \right)_{\Id}, \r{E} \tr \r{E}^* ) \\
\Hom_{R \tr S}(  _RS_S , _RS_S) & \Hom_R(R, {}_RS_S \tr_SS_R) \\};
\path[->,font=\scriptsize]
(m-1-1) edge (m-1-2)
        edge (m-2-1)
(m-2-1) edge node[auto]{$\varphi_0$}(m-2-2)
(m-1-2) edge node[right]{$\zeta_0$} (m-2-2);
\end{tikzpicture}
\end{center}
where $\varphi_0$ is given by the adjunction $\left(- \tr _RS_S\right) \ \dashv \ \left(- \tr _SS_R\right) = (-)_R$. Hence ${\zeta_0(i_0) = \varphi_0(\Id_{{}_RS_S}): 1_R \mapsto 1_S \tr 1_S}$ and this morphism indeed coincides with $\eta_0$. Similarly the existence of the dual bases $(e_i)_i$, $(f_j)_j$ implies there is an adjunction\\
${- \tr _SS_R = (-)_R \ \dashv  (-) \tr _RS_S}$ given by
\[ \varphi_1: \Hom_R(M \tr_SS_R, N) \mor \Hom_S(M,N \tr_RS_S): \psi \mapsto \left( \psi': m \mapsto \sum_i \psi(me_i) \tr f_i \right) \]
Where we used the lemma \ref{lem:centralelement} below to show that the morphisms in the image of $\varphi_1$ indeed have an $S$-module structure. A commutative diagram as above shows that ${\zeta_1(i_1) =\varphi_1(\Id_{{}_SS_R}): 1_S \mapsto \sum_i e_i \tr f_i}$ which coincides with $\eta_1$.

\end{proof}

\begin{lemma} \label{lem:centralelement} $\sum_i e_i\tr f_i$ is central in the $S$-bimodule $S\tr_R S$. I.e. for all $a\in S$ we have
\[
\sum_i ae_i\tr f_i=\sum_i e_i\tr f_ia
\]
\end{lemma}
\begin{proof} It is sufficient to prove that for all $j,k$ we have
\[
\sum_i \lambda(ae_if_j)\lambda(f_ie_k)=\sum_i \lambda(e_if_j)\lambda(f_iae_k)
\]
which is clear since both sides are equal to 
$
\lambda(ae_kf_j)
$.
\end{proof}

\subsection{Proof of theorem \ref{thm:noeth}}
We will now combine everything. As $X$ and $Y$ are noetherian we know that $\Qcoh(X)$ and $\Qcoh(Y)$ are locally noetherian categories and hence there exist collections of noetherian generating objects for these categories, say $\r{N}^X := \{ \r{N}_\alpha^X \}$ and $\r{N}^Y := \{ \r{N}_\beta^Y \}$. For each $i \in \d{Z}$  we define $\r{N}^n$ in $\Qcoh(X_n)$ as:
\[ \r{N}^n = \left\{ \begin{array}{cl} \r{N}^X & \textrm{ if $n$ is even} \\ \r{N}^Y & \textrm{ if $n$ is odd } \end{array} \right. \]
We shall prove that the collection 
\begin{equation} \label{eq:generatingset} \{ \r{N} \tr e_n \d{S}(\r{E}) \mid n \in \d{Z}, \r{N} \in \r{N}^n \} \end{equation}
forms a set of noetherian generators for $\Gr(\d{S}(\r{E}))$. Note that the collection is easily seen to generate as for each $\r{M} \in \Gr(\r{A})$ there is a surjective morphism 
\[ \DS_{i \in \d{Z}} \r{M}_n \tr e_i \r{A} \epi \r{M} \] and for each $n \in \d{Z}$ there is a surjective morphism 
\[\DS_{\alpha} (\r{N}^n_\alpha)^{m_\alpha} \epi \r{M}_n \]
where $\r{N}^i_\alpha \in \r{N}^i$. Hence we only need to show that the elements of (\ref{eq:generatingset}) are noetherian objects in $\Gr(\d{S}(\r{E}))$. By lemma \ref{lem:opencover} and Corollary \ref{cor:openrestriction} this can be checked locally for any open cover $X = \bigcup_l U_l$. By theorem \ref{thm:coverconditions} we may hence assume that $X = \Spec(R)$ and $Y=\Spec(S)$ are smooth affine varieties such that
\begin{enumerate}[i)]
\item $\r{L} \cong \r{O}_Y \cong \omega_Y$
\item $\omega_X \cong \r{O}_X$
\item $S/R$ is relative Frobenius of rank 4.
\end{enumerate} 
With these assumptions there are functors
\begin{equation} \label{eq:functors}
\begin{tikzpicture}
\matrix(m)[matrix of math nodes,
row sep=2em, column sep=2em,
text height=1.5ex, text depth=0.25ex]
{ \Gr(\d{S}(\r{E})) \\
 \Gr(\Gamma \left( \d{S}(\r{E}) \right)) \\
\Gr \left( \overline{\Gamma \left( \d{S}(\r{E}) \right)}\right) \\
\Gr(\Pi_R(S))\\};
\path[->,font=\scriptsize]
(m-1-1) edge node[left]{$\cong$}node[right]{$ \Gamma(-)$} (m-2-1)
(m-3-1) edge node[left]{$\cong$} node[right]{lemma \ref{lem:locallypirs}} (m-4-1);
\path[right hook->,font=\scriptsize]
(m-2-1) edge node[right]{Proposition \ref{prp:periodicalgebras}}(m-3-1);
\end{tikzpicture}
\end{equation}
Let $F:  \Gr(\d{S}(\r{E})) \mor \Gr(\Pi_R(S))$ be the composition. Then the above diagram shows that $F$ is an exact embedding of categories. Hence $\r{N} \tr e_n \d{S}(\r{E})$ is a noetherian object in $\Gr(\d{S}(\r{E}))$ if $F(\r{N} \tr e_n \d{S}(\r{E}))$ is a noetherian object in $\Gr(\Pi_R(S))$. On the other hand, as $\r{N}$ is noetherian in $\Qcoh(X_n)$ there is an $m \in \d{N}$ and a surjection $\r{O}_{X_n}^{\oplus m} \epi \r{N}$ giving rise to an surjection
\[ F(\r{O}_{X_n} \tr e_n \d{S}(\r{E}))^{\oplus m} \epi F(\r{N} \tr e_n \d{S}(\r{E})) \]
Hence it suffices to show that $F(\r{O}_{X_n} \tr e_n \d{S}(\r{E}))$ is a Noetherian object in $\Pi_R(S)$. This is easily seen since
\[ F(\r{O}_{X_n} \tr e_n \d{S}(\r{E})) = \begin{cases} R \cdot \Pi_R(S)(-n) & \textrm{if $n$ is even} \\ S \cdot \Pi_R(S)(-n) & \textrm{if $n$ is odd} \end{cases} \]
As both $R \cdot \Pi_R(S)$ and $S \cdot \Pi_R(S)$ are direct summands of $\Pi_R(S)$, which is a noetherian ring by Theorem \ref{thm:pirsnoetherian}, we have proven the theorem. \qed

\section{The Homological Properties of $\d{S}(\r{E})$}

\subsection{A Formula for $\Ext$-Groups}
\label{subsec:formula}
As before, in this section $\r{E}$ will be a locally free $X-Y$-bimodule of rank $(4,1)$ and we let $\r{A} := \d{S}(\r{E})$ denote the associated symmetric sheaf $\d{Z}$-algebra in standard form (see convention \ref{conv:conditions}). This section is dedicated to adapting the results in \cite{Mori07}, \cite{Nyman04}, \cite{Nyman09} and \cite{VdB_12} and to obtain a formula  for the $\Ext$-groups of pulled back sheaves on $\Proj(\r{A})$ along the truncation functors described in Subsection \ref{subsec:truncation}. To keep the geometric intuition we denote the functors $\omega \circ (-)_m:\Proj(\r{A})\mor \Qcoh(X_m)$ by ${\Pi_m}_*$ (compare with Theorem \ref{trm:commutativepushforward}). The left adjoints, which are given explicitly by $p((-)\tr e_m\r{A})$ following (\ref{def:e_nA}) and Definition \ref{def:proj}, are in turn denoted by $\Pi_m^*$. We shall use the notations $X_n$ and $Q_n$ as in Convention \ref{conv:Z_n} and  (\ref{eq:standardsheafZalgebra}).

\medskip

If $\r{E} \in \bimod(X-X)$ is locally free of rank (2,2) and $\r{A} = \d{S}(\r{E})$, \cite{Mori07} computes the Euler characteristics $\bl{\Pi_m^*\r{F}}{\Pi_n^*\r{G}}$ for two locally free sheaves $\r{F}$ and $\r{G}$ on $X$. In this section, we perform an analogous calculation in our setting where the bimodule $\r{E} \in \bimod(X-Y)$ is of rank (4,1). Motivated by Proposition \ref{prop:2-periodicity} our focus lies on $\bl{\Pi_m^*\r{F}}{\Pi_n^*\r{G}}$ with $\vert n-m\vert  \leq 1$.  This section is dedicated to proving the following slightly more general statement:

\begin{theorem}\label{thm:extandpullback}
Let $\r{E} \in \bimod(X,Y)$ be locally free of rank (4,1). Let $\r{F}$ and $\r{G}$ be locally free sheaves on $X_m$ respectively $X_n$ for $m,n \in \d{Z}$ such that $m \geq n-1$. Then

\[ \Ext^i_{\Proj(\r{A})} \left( \Pi_m^*\r{F}, \Pi_n^*\r{G} \right) \cong \Ext^i_{X_m} \left( \r{F}, \r{G} \tr \d{S}(\r{E})_{n,m} \right) \]
for all $i \geq 0$.
\end{theorem}
We have the following immediate corollary

\begin{corollary}

With the above assumptions, one has
\[	
\bl{\Pi_m^*\r{F}}{\Pi_n^*\r{G}}    = \bl{\r{F}}{ \r{G} \tr \d{S}(\r{E})_{n,m} }
\]

\end{corollary}

The proof of Theorem \ref{thm:extandpullback} is based on the existence of an exact sequence (see (\ref{eq:locallyfree}) below). To this end, we consider $\Theta_m$ defined by
\[
(\Theta_m)_n=
\begin{cases}
	0  & m\neq n\\
	\r{O}_{X_m} & n=m
\end{cases}
\]
\begin{remark}
Note that $\Theta_m$ is a right $\r{A}$-module using  $\r{A}_{i,i}=\r{O}_{X_i}$
\end{remark}
\begin{theorem}\label{trm:seqmain}
For each $m$, there is an exact sequence of locally free $\bimod(\r{O}_{X_m}-\r{A})$-bimodules\footnote{ See \cite[Section 3.2.]{VdB_12} for the definition of this category}
\begin{equation}
 0 \mor \r{Q}_m\tr e_{m+2}\r{A} \mor  \r{E}^{*m}\tr e_{m+1}\r{A} \mor e_m\r{A} \mor \Theta_m \mor 0 
\label{eq:locallyfree}
\end{equation}
\end{theorem}
\begin{proof}
By the nature of the relations this sequence is known to be right exact. The proof of the left exactness uses so-called `point  modules' and is deferred to Section \ref{sec:pointmodules}.
\end{proof}
As an immediate corollary of this theorem and its proof we find:
\begin{corollary}\label{cor:locallyfree}
for each $i,j \in \d{Z}$, the bimodule $\r{A}_{i,j}$ is locally free both on the left and on the right. The ranks are given by 
\[  \rk(\r{A})_{i,j}:=
\begin{cases}
(j-i+1,j-i+1) & i\equiv j \, mod \,2\\
\displaystyle \left(\frac{j-i+1}{2}, 2(j-i+1)\right) & i\, odd, \, j \, even\\
\displaystyle \left( 2(j-i+1), \frac{j-i+1}{2} \right) & i\, even, \, j \, odd
\end{cases}
\]
\begin{proof}
We have $\rk(\r{E})=(4,1)$ and $\rk(\r{E}^*)=(1,4)$, $\rk(\r{Q}_m)=(1,1)$. Since the rank is additive on short exact sequence, one can now verify the claim by induction in the three cases on $m$ using the sequences in (\ref{eq:locallyfree}).
\end{proof}

\end{corollary}
This result in turn implies the following convenient fact
\begin{lemma}\label{lem:pushforwardexact}
For each $m \in \d{Z}$, $\Pi_m^*: \Qcoh(X_m) \mor \Proj(\r{A})$ is an exact functor
\end{lemma}

\begin{proof}
For each $n \geq m$, $\r{A}_{m,n}$ is locally free by Corollary \ref{cor:locallyfree}, hence the functor $- \tr \r{A}_{m,n}: \Qcoh(X_m) \mor \Qcoh(X_n)$ is exact.
\end{proof}
As an example application of the above lemma, we mention the following adjunction formula:
\begin{lemma} \label{cor:differentRHom}
There is a natural isomorphism for all $\r{F} \in \Qcoh(X_m)$ and ${\r{M} \in \r{D}^+(\Proj(\r{A}))}$:
\[ \RHom_{\Proj(\r{A})}(\Pi^*_m \r{F},\r{M}) \cong \RHom_{X_m}(\r{F},\der{\Pi_m}_* \r{M}) \]
\end{lemma}

\begin{proof}
Since  $\Pi^*_m$ is an exact left adjoint to $\Pi_{m,*}$, the latter must preserve injective objects and the result follows.
\end{proof}

For the purposes of proving Theorem \ref{thm:extandpullback} we are especially interested in the case where $\r{M}=\Pi_n^* \r{G}$ for a locally free sheaf $\r{G}$ on $X_n$. It follows that we need to understand complexes of the form $R{\Pi_m}_*(\Pi_n^*\r{G})$. The strategy for computing the homology of this complex is as follows: by Lemma \ref{lem:pushforward/tau} below, it suffices to give a description the derived functors of the torsion functor $\tau$. These in turn follow from the derived functors of an internal Hom-functor $\underline{\ShHom}$ constructed in \cite{Nyman04} (Lemma \ref{lem:taushhom}).

\begin{lemma}\label{lem:tau/omega} We have the following facts for the derived functors of the torsion functor $\tau: \Gr(\r{A}) \mor \Tors(\r{A})$:
\begin{enumerate}[i)]
\item for $i \ge 1$, there is an isomorphism of functors
\[ \der^{i+1}\tau \cong (\der^i \omega) \circ p \]
\item For each $\r{M} \in \Gr(\r{A})$ there is an exact sequence:
\[ 0 \mor \tau(\r{M}) \mor \r{M} \mor \omega(p(\r{M})) \mor \der^1 \tau(\r{M}) \mor 0 \]
\end{enumerate}
\end{lemma}

\begin{proof}
By Theorem \ref{thm:noeth},  $\Gr(\r{A})$ is a locally noetherian category. Moreover, by \cite[Lemma 2.12]{Nyman09}, any essential extension of a torsion module remains a torsion module. In particular, the category $\Tors(\r{A})$ is closed under injective envelopes, the result now follows from \cite[Theorem 2.14.15]{Smith_99}.
\end{proof}

\begin{lemma}\label{lem:pushforward/tau}
For $i\ge 1$, and $\r{F} \in \Qcoh(X_m)$ there is an isomorphism
$$\der^i{\Pi_m}_*(\Pi_n^*\r{F}) \cong \der^{i+1}\tau(\r{F} \tr e_n \r{A})_m$$
\end{lemma}

\begin{proof}
As the functors $p$ and $(-)_m$ are exact there is a functorial isomorphism 
\begin{equation}
\label{eq:pushforward/tau}
(\der^i {\Pi_m}_*)(p)(-) \cong \der^i\omega(p(-))_m
\end{equation}
Combining this isomorphism with the one in lemma \ref{lem:tau/omega} we obtain for each $i \geq 1:$
\[ \der^i{\Pi_m}_*(\Pi_n^*\r{\r{F}}) :=  \der^i{\Pi_m}_*(p(\r{F} \tr e_n \r{A})) \cong \der^i \omega(p(\r{F} \tr e_n \r{A}))_m \cong \der^{i+1}\tau(\r{F} \tr e_n \r{A})_m \qedhere
\]
\end{proof}
\noindent
The following is based on \cite[Section 3.2]{Nyman04}:\\
Let $\BiMod(\r{A}- \r{A})$ denote the category whose objects are of the form $${\{ \r{B}_{m,n}\in \BiMod(X_m-X_n) \}_{m,n}}$$ such that the left and right multiplications 
\begin{eqnarray*}
&\r{A}_{l,m} \tr \r{B}_{m,n} \mor \r{B}_{l,n}\text{ and } 
\r{B}_{m,n} \tr \r{A}_{n,l} \mor \r{B}_{m,l}
\end{eqnarray*} 
 are compatible in the obvious sense. We denote by $\d{B}$ for the subcategory for which all $\r{B}_{m,n}$ are coherent and locally free. There are $\Hom$-functors
\begin{eqnarray*}
&\underline{\ShHom}:\d{B}^{op} \times \Gr(\r{A}) \mor \Gr(\r{A}) \text{ and }\\
&\ShHom: \BiMod(\r{O}_{X_n}-\r{A}) \times \Gr(\r{A}) \mor \Qcoh(X_n)
 \end{eqnarray*}
satisfying the following properties:

\begin{proposition}\label{prp:sheafhom}
\begin{enumerate}[i)]
\item $\underline{\ShHom}(\r{B},\r{M})_m = \ShHom(e_m  \r{B}, \r{M})$ for all $\r{B} \in \d{B}$ and $\r{M} \in \Gr(\r{A})$
\item $\underline{\ShHom}: \d{B}^{op} \times \Gr(\r{A}) \mor \Gr(\r{A}$) is a bifunctor, left exact in both its arguments
\item $\ShHom: \BiMod(\r{O}_{X_n}-\r{A}) \times \Gr(\r{A}) \mor \Qcoh(X_n)$ is a bifunctor, left exact in both its arguments
\item $\ShHom(\r{Q} \tr e_m \r{A}, \r{M}) \cong \r{M}_m \tr \r{Q}^*$ for all $\r{M} \in \Gr(\r{A})$ and locally free $X_m$-bimodules $\r{Q}$
\end{enumerate}
\end{proposition}
\begin{proof}
\begin{enumerate}
\item[\emph{i)}] This follows immediately by checking the precise definitions in \cite[\S 3.2]{Nyman04}
\item[\emph{ii)}] see \cite[Proposition 3.11, Theorem 3.16(1)]{Nyman04}
\item[\emph{iii)}] see \cite[Theorem 3.16(3)]{Nyman04}
\item[\emph{iv)}] see \cite[Theorem 3.16(4)]{Nyman04} \qedhere
\end{enumerate}
\end{proof}

By \emph{ii)} and \emph{iii)} in the above proposition one can define the right derived functors $\underline{\ShExt}^{i}$ and $\ShExt^{i}$ for all $i \geq 0$. Moreover we use the notation $\r{A}_{\geq l}$ to denote the object in $\d{B}$ given by 
\[ \left( \r{A}_{\geq l} \right) _{m,n} = \left\{ \begin{array}{cc} \r{A}_{m,n} & \textrm{ if $n-m \geq l$} \\ 0 & \textrm{ else} \end{array} \right. \]
and $\r{A}_0 := \r{A} / \r{A}_{\geq 1}$. Then we have the following relation between the derived functors of $\tau$ and $\underline{\ShExt}^{i}$:
\begin{lemma}\label{lem:taushhom}
$\displaystyle R^i\tau(-)\cong \lim_{l \to \infty} \underline{\ShExt}^{i}_{\Gr({\r{A})}}(\r{A}/\r{A}_{\ge l},- )$
\end{lemma}
\begin{proof}
By \cite[Proposition 3.19]{Nyman09}, we have an isomorphism of functors
\[ \tau \cong \lim_{l \to \infty} \underline{\ShHom}_{\Gr( \r{A})}( \r{A}/\r{A}_{\ge l},-) \]
Applying this to the injective resolution and subsequently taking homology yields the required result
\end{proof}

\begin{lemma} \label{lem:concentrateddegreel}
Let $\r{B} \in \d{B}$ be concentrated in degree $l\geq 0$ (i.e. $\r{B}_{m,n}=0$ whenever $m+l \neq n$) and $\r{V}$ a locally free sheaf. Then for $n-l-1 \le m$ and for all $i \geq 0$:
\[ \underline{\ShExt}^{i}(\r{B}, \r{V}\tr e_n\r{A})_m=0 \]
\end{lemma}
\begin{proof}
By \cite[cor. 4.6]{Nyman04}, there is an isomorphism
\[
\underline{\ShExt}^{i}(\r{B}, \r{V} \tr e_n \r{A})_m  \cong \underline{\ShExt}^{i}(\r{A}_0, \r{V} \tr e_n \r{A})_{m+l} \tr \r{B}_{m,m+l}^*
\]
which easily reduces the proof to the case $\r{B}=\r{A}_0$ and in particular $l=0$.\\
By Proposition \ref{prp:sheafhom}\emph{(iv)} we see that the exact sequence from Theorem \ref{trm:seqmain} forms a resolution of $e_m\r{A}_0 = \Theta_m$ through $\ShHom(-,\r{V} \tr e_n\r{A})$-acyclic sheaves. In particular we can calculate  $\underline{\ShExt}^{i}(\r{A}_0, \r{V} \tr e_n \r{A})_m = \ShExt^{i}(e_m A_0, \r{V} \tr e_n \r{A})$ by taking homology of the complex
\begin{align*}
&0\mor \ShHom(e_m \r{A},\r{V}\tr e_n\r{A}) \stackrel{d_0}{\mor} \ShHom(\r{E}^{*m}\tr e_{m+1}\r{A}, \r{V}\tr e_n\r{A})\\
&\stackrel{d_1}{\mor} \ShHom(\r{Q}_m\tr e_{m+2}\r{A},\r{V}\tr e_n\r{A})\mor  0
\end{align*}
using Proposition \ref{prp:sheafhom}\emph{(iv)}, this complex becomes
\begin{equation} 
\label{eq:concentrated} 0\mor \r{V}\tr \r{A}_{n,m}\stackrel{d_0}{\mor}\r{V}\tr \r{A}_{n,m+1}\tr\r{E}^{*m+1} \stackrel{d_1}{\mor} \r{V}\tr \r{A}_{n,m+2}\tr \r{Q}_m^*\mor 0 
\end{equation}
Hence we have
\begin{itemize}
\item $\ShExt^{0}(e_m \r{A}_0, \r{V} \tr e_n \r{A}) =  \ker (d_0)$
\item $\ShExt^{1}(e_m \r{A}_0, \r{V} \tr e_n \r{A}) = \ker (d_1)/\im(d_0)$
\item $\ShExt^{2}(e_m \r{A}_0, \r{V} \tr e_n \r{A}) = \coker(d_1)$
\item $\ShExt^{i}(e_m \r{A}_0, \r{V} \tr e_n \r{A}) = 0$ for all $i \geq 3$
\end{itemize}
To show the exactness of (\ref{eq:concentrated}), we first note that the explicit nature of the isomorphisms in \cite{Nyman04} yield that (\ref{eq:concentrated}) is obtained from the sequence
\begin{equation}
\label{eq:concentratedstep1}
0\mor \r{A}_{n,m}\mor \r{A}_{n,m+1}\tr \r{E}^{*m+1}\mor \r{A}_{n,m+2}\tr Q_m^*\mor 0
\end{equation}
by tensoring with $\r{V}$. Since $\r{V}$ is locally free, it preserves exactness and it suffices to verify that (\ref{eq:concentratedstep1}) is exact. Next, we tensor with the invertible bimodule $Q_m$ to obtain
\begin{equation} \label{eq:concentratedstep2} 
0\mor \r{A}_{n,m}\tr Q_m\stackrel{d_0}{\mor}  \r{A}_{n,m+1}\tr\r{E}^{*m+1} \tr Q_m\stackrel{d_1}{\mor} \r{A}_{n,m+2}\mor 0 
\end{equation}
We can replace the middle term in (\ref{eq:concentratedstep2}) to obtain:
\begin{equation} 
\label{eq:concentratedstep3} 
0\mor \r{A}_{n,m}\tr Q_m\stackrel{d_0}{\mor}  \r{A}_{n,m+1}\tr\r{E}^{*m+1} \stackrel{d_1}{\mor} \r{A}_{n,m+2}\mor 0 
\end{equation}
A similar but tedious computation as in \cite[\S 7.5]{Nyman04} shows that this sequence  coincides  with the exact sequence in Theorem $\ref{trm:seqmain}$ in degree $n$ for left modules. We conclude the result by the same argument as for Theorem \ref{trm:seqmain}.
\end{proof}

\begin{lemma}\label{lem:induction2}
\begin{center}
$\underline{\ShExt}^i(\r{A}/\r{A}_{\ge l}, \r{V} \tr e_n\r{A})_m=0 \, \, \textrm{for } 
m \geq n-1 \textrm{ and } i \geq 0
$
\end{center}
\end{lemma}

\begin{proof} \
Consider the short exact sequence 
\[ 0 \mor \r{A}_{\geq l}/ \r{A}_{\geq l+1} \mor \r{A}/ \r{A}_{\geq l+1} \mor \r{A}/ \r{A}_{\geq l} \mor 0 \]
Applying $\underline{\ShHom}(-, \r{V} \tr e_n \r{A})$ gives rise to a long exact sequence for each $m \geq n-1$
\begin{eqnarray*} \ldots \mor \underline{\ShExt}^{i}(\r{A}_{\geq l}/ \r{A}_{\geq l+1}, \r{V} \tr e_n \r{A})_m & \mor & \underline{\ShExt}^{i}(\r{A}/ \r{A}_{\geq l+1}, \r{V} \tr e_n \r{A})_m\\
 \mor \ \ \underline{\ShExt}^{i}(\r{A}/ \r{A}_{\geq l}, \r{V} \tr e_n \r{A})_m & \mor & \underline{\ShExt}^{i+1}(\r{A}_{\geq l}/ \r{A}_{\geq l+1}, \r{V} \tr e_n \r{A})_m \mor \ldots \end{eqnarray*}
As $m \geq n-1$ it follows from Lemma \ref{lem:concentrateddegreel} that for each $i \geq 0$ we have an exact sequence
\[ 0 \mor  \underline{\ShExt}^{i}(\r{A}/ \r{A}_{\geq l+1}, \r{V} \tr e_n \r{A})_m \mor \underline{\ShExt}^{i}(\r{A}/ \r{A}_{\geq l}, \r{V} \tr e_n \r{A})_m  \mor 0 \]
Hence 
\[ \underline{\ShExt}^{i}(\r{A}/ \r{A}_{\geq l}, \r{V} \tr e_n \r{A})_m \cong \underline{\ShExt}^{i}(\r{A}/ \r{A}_{\geq 0}, \r{V} \tr e_n \r{A})_m = \underline{\ShExt}^{i}(0, \r{V} \tr e_n \r{A})_m = 0 \]
\end{proof}
We can now finish the proof of Theorem \ref{thm:extandpullback}.
\begin{proof} \emph{of Theorem \ref{thm:extandpullback}}\\
Take $m,n \in \d{Z}$ with $m \geq n-1$. Let $\r{F}$ be locally free on $X_m$ and $\r{G}$ locally free on $X_n$, then by Corollary \ref{cor:differentRHom}:
\begin{eqnarray*}
 \Ext^i_{\Proj(\r{A})} \left( \Pi_m^*\r{F}, \Pi_n^*\r{G} \right) & = & h^i \left( \der \Hom_{\Proj(\r{A})} \left( \Pi_m^*\r{F}, \Pi_n^*\r{G} \right) \right) \\
 & \cong & h^i \left( \der \Hom_{X_m} \left( \r{F}, \der {\Pi_m}_*\Pi_n^*\r{G} \right) \right)
\end{eqnarray*}
Now for $i \geq 1$ we have 
\begin{eqnarray*}
\der^i {\Pi_m}_*\Pi_n^*\r{G} & \cong & \der^{i+1} \tau (\r{G} \tr e_n \r{A})_m \\
 & \cong & \lim_{l \rightarrow \infty} \underline{\ShExt}^{i+1}(\r{A}/\r{A}_{\geq l}, \r{G} \tr e_n \r{A})_m \\
 & = & 0
\end{eqnarray*}
by Lemmas \ref{lem:pushforward/tau}, \ref{lem:taushhom} and \ref{lem:induction2} respectively.\\
In particular the complex $\der {\Pi_m}_*\Pi_n^*\r{G}$ is quasi-isomorphic to the complex that is equal to ${\Pi_m}_*\Pi_n^*\r{G}$ concentrated in position zero. Finally we can conclude by noticing that ${\Pi_m}_*\Pi_n^*\r{G} = \left( \omega p (\r{G} \tr e_n \r{A}) \right)_m$ and by Lemma \ref{lem:tau/omega} there is an exact sequence
\[ 0 = \tau(\r{G} \tr e_n \r{A})_m \mor \r{G} \tr \r{A}_{n,m} \stackrel{\cong}{\mor} \omega(p(\r{G} \tr e_n \r{A}))_m \mor \der^1 \tau(\r{G} \tr e_n \r{A})_m = 0 \]
where the first term equals zero because $\r{G} \tr e_n \r{A}$ is torsion free and the last term is zero because $\displaystyle \der^1 \tau(\r{G} \tr e_n \r{A})_m \cong \lim_{l \rightarrow \infty} \underline{\ShExt}^{\ 1}(\r{A}/\r{A}_{\geq l}, \r{G} \tr e_n \r{A})_m = 0$.\\
Hence we can conclude that for $m \geq n-1$ we have
\begin{eqnarray*}
\Ext^i_{\Proj(\r{A})} \left( \Pi_m^*\r{F}, \Pi_n^*\r{G} \right) & \cong & h^i \left( \der \Hom_{X_m} \left( \r{F}, \der {\Pi_m}_*\Pi_n^*\r{G} \right) \right) \\
& \cong & h^i \left( \der \Hom_{X_m} \left( \r{F}, \r{G} \tr \r{A}_{n,m} \right) \right) \\
& = & \Ext^i_{X_m} \left( \r{F}, \r{G} \tr \r{A}_{n,m} \right) 
\end{eqnarray*}
\end{proof}
\subsection{Point Modules in the Rank $(4,1)$ Case}
\label{sec:pointmodules}
We remain in the setting where $\r{A}=\d{S}(\r{E})$ denotes a symmetric sheaf $\mathbb{Z}$-algebra in standard form with ${\r{E} \in \bimod(X-Y)}$ locally free of rank (4,1), given in the form of $\r{E} = {}_f (\r{L})_{\Id}$ for a finite morphism $f: Y \mor X$ of degree 4 as in Lemma \ref{lem:inducedbyline}. Denote by $\alpha: X \mor \Spec(\k)$ and $\beta: Y \mor \Spec(\k)$ be the structure morphisms. Extending our convention \ref{conv:Z_n} we will write
\[ (X_n, \alpha_n) = \left\{ \begin{array}{cc} (X,\alpha) & \textrm{if $n$ is even} \\
 (Y,\beta) & \textrm{if $n$ is odd} \end{array} \right. \]
We say $P_n \in \coh(X_n)$ is locally free over $\k$ of rank $l$ if the support of $P_n$ is finite over $\k$ and $\dim_\k (\alpha_{n,*} P_n)=l$.\\ \\
A module $P \in \Gr(\r{A})$ is said to be generated in degree $m$ if $P_n = 0$ for all $n<m$ and $P_m \tr \r{A}_{m,n} \mor P_n$ is surjective for all $n \geq m$. 
As $\r{A}$ is generated in degree one as an algebra, we have surjectivity of $P_{n_1} \tr \r{A}_{n_1,n_2} \mor P_{n_2}$ for all $n_2 \geq n_1 \geq m$ by the following commuting diagram
\begin{center}
\begin{tikzpicture}
\matrix(m)[matrix of math nodes,
row sep=3em, column sep=3em,
text height=1.5ex, text depth=0.25ex]
{ P_m \tr \r{A}_{m,n_1} \tr \r{A}_{n_1,n_2} & P_{n_1} \tr \r{A}_{n_1,n_2} \\
P_m \tr \r{A}_{m,n_2} & P_{n_2} \\};
\path[->,font=\scriptsize]
(m-1-2) edge (m-2-2);
\path[->>]
(m-1-1) edge (m-1-2)
        edge (m-2-1)
(m-2-1) edge (m-2-2);
\end{tikzpicture}
\end{center}
\begin{remark} An obvious example of a module generated in degree $m$ is $e_m \r{A}$. The above diagram implies that the maps $\r{A}_{m,n} \tr e_n \r{A} \mor e_m \r{A}$ are surjective for all $m \geq n$.
\end{remark}
An $m$-shifted point-module over $\r{A}$ is defined in \cite{VdB_12} as an object $P \in \Gr(\r{A})$ such that $P$ is generated
in degree $m$ and for which $P_n$ is locally free of rank one over $\k$ for all $n \geq m$. As the next lemma shows, this concept is not very useful in our setting:

\begin{lemma} \label{lem:notpointmod}
Let $i \in \d{Z}$ and $P \in \Gr(\r{A})$ generated in degree $2i$ such that $P_{2i}$ and $P_{2i+1}$ are locally free of rank one over $\k$. Then $P_n=0$ for all $n \geq 2i+2$.
\end{lemma}
\begin{proof}
Recall that the following composition
\[ 
P_{2i} \mor P_{2i} \tr \r{E}^{*2i} \tr \r{E}^{*2i+1} \mor P_{2i+1} \tr \r{E}^{*2i+1} \mor P_{2i+2}
 \] 
must be zero as it represents the action of $\r{Q}_{2i}$. By \cite[lemma 4.3.2.]{VdB_12} this composition equals
\[ 
P_{2i} \stackrel{\varphi_{2i}^*}{\mor} P_{2i+1} \tr \r{E}^{*2i+1} \stackrel{\varphi_{2i+1}}{\mor} P_{2i+2} 
\] 
where $\varphi_{2i}^*$ is obtained by adjointness from $\varphi_{2i}: P_{2i} \tr \r{E}^{*2i} \mor P_{2i+1}$ and $\r{E}^{*2i+1}$ has rank $(1,4)$. Since $P_{2i}$ and $P_{2i+1} \tr \r{E}^{*2i+1}$ are locally  free of rank one over $\k$ we obtain that  $\varphi_{2i}^*$ is either an isomorphism or zero. Similarly $\varphi_{2i+1}$ is either injective or zero. Hence the only way the composition can be zero is if $\varphi_{2i}^*=0$ or $\varphi_{2i+1}=0$. The first  doesn't occur as $\varphi_{2i} \neq 0$ (because $P$ is generated in degree $2i$ and $P_{2i+1}\neq 0$). Hence we  have $\varphi_{2i+1} = 0$. However $\varphi_{2i+1}$ is  surjective  (because $P$ is generated in degree $2i$), implying that  $P_{2i+2}=0$. Using surjectivity of $P_{2i+2} \tr \r{A}_{2i+2,n} \mor P_n$ for all $n \geq 2i+2$ the result follows.
\end{proof}
We thus propose the following variation of the above definition, better suited to our needs:
\begin{definition}
A shifted point module is an object $P \in \Gr(\r{A})$ which is generated in degree $2i$ for some integer $i$ such that for all $n \geq 2i$,
 $P_n$ is locally free over $\k$ of rank one if $n$ is even  and $P_n$ is locally free over $\k$ of rank two if $n$ is odd. We will often use the short hand notation $\dim_\k(P_n) = \dim_{\k}(\alpha_{n,*}(P_n))$ whenever the latter is finite. So we could say $P$ is a shifted point module if is generated in degree $2i$ and:
\[ \dim_\k(P_n) = \left\{ \begin{array}{cl} 0 & \textrm{if $n < 2i$} \\
1 & \textrm{if $n \geq 2i$ is even} \\
2 & \textrm{if $n > 2i$ is odd}
\end{array}
\right. \]
\end{definition}
The following lemma shows that this new definition of point modules is better behaved than the naive one:
\begin{lemma} \label{lem:pointupperbound}
Let $P \in \Gr(\r{A})$ be a graded module and $i \in \d{Z}$ such that:
\begin{itemize}
\item $P$ is generated in degree $2i$
\item $\dim_\k(P_{2i})=1$
\item $\dim_\k(P_{2i+1})=2$
\end{itemize}
Then for all $n \geq 2i+2$ fixed, we have
\begin{equation} \label{eq:pointabove}
 \dim_\k(P_n) \leq \left\{ \begin{array}{cl} 1 & \textrm{if $n$ is even} \\
2 & \textrm{if $n$ is odd}
\end{array}
\right. 
\end{equation}
Moreover if equality holds in \emph{(\ref{eq:pointabove})}, then $P_n$ is characterized up to unique isomorphism by the data $\varphi_{2i}: P_{2i} \tr \r{E}^{*2i} \mor P_{2i+1}$.\\
If on the other hand \emph{(\ref{eq:pointabove})} is a strict inequality for some $n$, then $P_l = 0$ for all $l > n$.
\end{lemma}
\begin{proof}
We prove all facts by induction on $n$. So suppose (\ref{eq:pointabove}) and the subsequent claims hold for $n=2i, \ldots , m$. We distinguish several cases depending on whether the inequalities are in fact equalities or not.\\ \\
\textbf{Case 1: Equality holds in (\ref{eq:pointabove}) for $n=2i, \ldots , m$. }\\
The following composition is zero:
\[ P_{m-1} \stackrel{\varphi_{m-1}^*}{\mor} P_m \tr \r{E}^{*m} \stackrel{\varphi_m}{\mor} P_{m+1} \]
$\varphi_m$ is surjective, using the fact that the ranks are $(4,1)$ or $(1,4)$ depending on the parity of $m$, one verifies that  (\ref{eq:pointabove}) holds for $n=m+1$ if  $\varphi_{m-1}^*$ is injective. Moreover the same reasoning shows that
 if the equality holds for $\dim_\k(P_{m+1})$, then $P_{m+1} \cong \coker(\varphi_{m-1}^*)$ and is hence defined up to unique isomorphism.\\ \\
\textbf{Case 1a: $m$ is odd}\\
We have $\dim_\k(P_{m-1})=1$ and the claim reduces to $\varphi_{m-1}^* \neq 0$  which holds as  $\varphi_{m-1} \neq 0$\\
\textbf{Case 1b: $m$ is even}\\
If $\varphi_{m-1}^*$ is not injective, then there exists a $W \subset P_{m-1}$ with $\dim_\k(W)=1$ such that the composition 
\[ W \hookrightarrow P_{m-1} \stackrel{\varphi_{m-1}^*}{\mor} P_m \tr \r{E}^{*m} 
\]
or equivalently the composition
\[
 W\tr \r{E}^{* m-1} \hookrightarrow P_{m-1}\tr \r{E}^{* m-1} \mor P_m 
\]
is zero. This implies that there is a $\overline{W} \in \Gr(\r{A})$ given by $\overline{W}_{m-1} = W$ and $\overline{W}_l = 0$ for $l \neq m-1$ together with an embedding $\chi : \overline{W} \hookrightarrow P_{\geq m-2}$. Let $C = \coker(\chi)_{\ge m-2}$. Then $C$ is generated in degree $m-2$ (which is even!) and $\deg_k(C_{m-2})=\deg_k(C_{m-1})=\deg_k(C_m)=1$ contradicting Lemma \ref{lem:notpointmod}.\\ \\
\textbf{Case 2: There is an integer $n \in \{ 2i+2, \ldots , m \}$ providing a  strict inequality for $\dim_\k(P_n)$ in (\ref{eq:pointabove})}\\
Let $n_0$ be the smallest such $n$. We have to show $P_l=0$ for all $l > n_0$. \\
Assume that $P_{n_0}=0$. Then $P_l=0$ by surjectivity of $P_{n_0} \tr \r{A}_{n_0,l} \mor P_l$.\\
The only nontrivial case is when $n_0$ is odd and $\dim_\k(P_{n_0})=1$. In this case $\dim_\k(P_{n_0 - 1})=1$ as well and the result follows from Lemma \ref{lem:notpointmod}.
\end{proof}
\begin{remark}
The proof of the above lemma also shows that any data ${\varphi_{2i}: P_{2i} \tr \r{E}^{*2i} \twoheadrightarrow P_{2i+1}}$ with $\dim_\k(P_{2i})=1$ and $\dim_\k(P_{2i+1})=2$ can be extended to a shifted point module which is unique up to unique isomorphism.
\end{remark}
From now on we use the following short hand notation:
\begin{equation}
\label{eq:lnp}
L_{n,p} := \r{O}_p \tr e_n \r{A}
\end{equation}

where $p$ is any point on $X_n$.
\begin{proof} \emph{of Theorem \ref{trm:seqmain}}\\
The exactness of the sequence (\ref{eq:locallyfree}) can be checked for each degree $n$ separately:
\begin{equation} \label{eq:locallyfreedegree} 0 \mor \r{Q}_m \tr \r{A}_{m+2,n} \mor \r{E}^{*m} \tr \r{A}_{m+1,n} \mor \r{A}_{m,n} \mor 0 \end{equation}
As all terms in this sequence are elements of $\bimod(X_m - X_n)$, applying the pushforward of the projection $\pi_m: X_m\times X_n\mor X_m$, yields a sequence of coherent sheaves on $X_m$:
\begin{equation} \label{eq:locallyfreeleft} 0 \mor \pi_{m,*}(\r{Q}_m \tr \r{A}_{m+2,n} )\mor \pi_{m,*} (\r{E}^{*m} \tr \r{A}_{m+1,n}) \mor \pi_{m,*} (\r{A}_{m,n}) \mor 0 \end{equation}
and (\ref{eq:locallyfreeleft}) is exact if and only if (\ref{eq:locallyfreedegree}) is since the support of these bimodules is finite. The structure of the relations on $\r{A}$ implies that (\ref{eq:locallyfree}) and hence also (\ref{eq:locallyfreedegree}) and (\ref{eq:locallyfreeleft}) are right exact. Now for any point $p \in X_m$ the following complex will be right exact as well:
\begin{eqnarray} \notag 0 \rightarrow \r{O}_p \tr \pi_{m,*} (\r{Q}_m \tr \r{A}_{m+2,n} ) & \rightarrow & \r{O}_p \tr \pi_{m,*} (\r{E}^{*m} \tr \r{A}_{m+1,n})  \rightarrow  \\
\label{eq:locallyfreestalklike1} & \rightarrow & \r{O}_p \tr \pi_{m,*} (\r{A}_{m,n}) \rightarrow 0. \end{eqnarray}
 As all terms (\ref{eq:locallyfreestalklike1}) are locally free over $\k$, its left exactness can be checked numerically. Hence in order to prove the lemma we show that the terms in (\ref{eq:locallyfreestalklike1}) have the ``correct'' constant dimension (see (\ref{eq:correctdimensions})) at each point $p\in X_m$. From this it follows that (\ref{eq:locallyfreedegree}) is exact and its terms are locally free on the left. The claim that the terms are locally free on the right in turn follows from \cite[Proposition 3.1.6]{VdB_12}.)\\[\medskipamount]
We are left with finding the length of the objects in (\ref{eq:locallyfreestalklike1}). Any object in $\bimod(X_m-X_n)$ is of the form ${}_u \r{U}_v$ for finite maps $u$ and $v$. As taking the direct image through a finite morphism preserves the length, for any such bimodule, we compute:
\begin{eqnarray*}
\dim_\k( \r{O}_p \tr \pi_{m,*}( {}_u \r{U}_v)) & = & \dim_\k( \r{O}_p \tr u_* \r{U}) \\
 & = & \dim_\k( u_*( u^*(\r{O}_p) \tr \r{U})) \\
 & = & \dim_\k( u^*(\r{O}_p) \tr \r{U}) \\
 & = & \dim_\k( v_*( u^*(\r{O}_p) \tr \r{U})) \\
 & = & \dim_\k( \r{O}_p \tr {}_u \r{U}_v))
\end{eqnarray*}
It follows that the length of the terms in (\ref{eq:locallyfreestalklike1}) can be computed from 
\begin{equation} 
\label{eq:locallyfreestalklike2} 
0 \rightarrow \r{O}_p \tr \r{Q}_m \tr \r{A}_{m+2,n}  \rightarrow \r{O}_p \tr \r{E}^{*m} \tr \r{A}_{m+1,n} \rightarrow \r{O}_p \tr \r{A}_{m,n} \rightarrow 0 
\end{equation}

In the case where $m=2i-1$, the fact that $\dim_\k(\r{O}_p\tr \r{E}^{* 2i-1})=1$, implies that there must be a point $q \in X_{2i}$ such that $\r{O}_{p} \tr \r{E}^{*2i-1} = \r{O}_{q}$. Similarly, in the case where  $m=2i$, we have  $\dim_\k(\r{O}_{p} \tr \r{E}^{*2i})=4$, and  there must be points $\widetilde{q^a} \in X_{2i+1}$, $a=1,\ldots, 4$ such that $\r{O}_{p} \tr \r{E}^{*2i}$ is an extension of the $\r{O}_{\widetilde{q^a}}$. Put 
\[
M_{2i+1,p}=
\r{O}_p \tr_{X_{2i}} \r{E}^{*2i} \tr_{X_{2i+1}} e_{2i+1} \r{A}.
\]
 Then $M_{2i+1,p}$ is an extension of the $L_{2i+1,\widetilde{q^a}}$. The sequence (\ref{eq:locallyfreestalklike2})  now gives rise to the following right exact sequences
\begin{equation} \label{eq:locallyfreeloceven}
 L_{2i+1, p} \mor L_{2i,q} \mor L_{2i-1,p} \mor 0 
\end{equation} \begin{equation} \label{eq:locallyfreelocodd}
 L_{2i+2, p} \mor M_{2i+1,p} \mor L_{2i,p} \mor 0
\end{equation}
Finally there also is a right exact sequence:
\begin{equation} \label{eq:pointresolution}
L_{2i+1, p'} \mor L_{2i-1,p} \mor P_p \mor 0
\end{equation}
where the morphism $L_{2i+1, p'} \mor L_{2i-1,p}$ comes from the fact that $\dim_\k(\r{O}_p \tr \r{A}_{2i-1,2i+1})=3>0$ so that there is a $p' \in X_{2i+1}$ with a nonzero morphism $\r{O}_{p'} \mor \r{O}_p \tr \r{A}_{2i-1,2i+1}$. $P_p$ is defined as the cokernel of this morphism.\\ \\
We now prove the following by induction on $j$ (simultaneously for all points $p$ and all $i \in \d{Z}$):
\begin{eqnarray}
\notag  \dim_\k((P_p)_{2i+2j})	&=&1 \\
\notag \dim_\k((P_p)_{2i+2j+1})&=&2\\
\label{eq:correctdimensions} \dim_\k((L_{2i,p})_{2i+2j})&=&2j+1\\
\notag \dim_\k((L_{2i,p})_{2i+2j+1})&=&4j+4\\
\notag \dim_\k((L_{2i-1,p})_{2i+2j})&=&j+1\\
\notag \dim_\k((L_{2i-1,p})_{2i+2j+1})&=&2j+3
\end{eqnarray}
It is easy to see that these claims hold for $j=0$. So by induction we  suppose they hold for $j=0, \ldots, l$, for all $p$ and for all $i \in \d{Z}$. We prove that the claims also hold for $j=l+1$.\\ \\
By (\ref{eq:locallyfreelocodd}) we see: 
\begin{eqnarray*}
\dim_\k((L_{2i,p})_{2i+2l+2}) & \geq & \dim_\k((M_{2i+1,p})_{2i+2l+2})-\dim_\k((L_{2i+1,p})_{2i+2l+2}) \\
& = & \sum_{a=1}^4 \dim_\k((L_{2i+2,\widetilde{q^a_{2i+1}}})_{2i+2l+2})-\dim_\k((L_{(2i+2,p})_{2i+2l+2}) \\
& = & 4 \cdot (l+1) - (2l+1) \\
& = & 2l+3
\end{eqnarray*}
where the last equality follows from the induction hypothesis. This can be written schematically as:
\begin{equation} \label{eq:firstbound} \begin{array}{ccccccccccc}
&0& & & & & & & & & \\
& \uparrow & & & & & & & & & \\
& L_{2i,p} & 0 & 1 & 4 & 3 & \ldots & 2l+1 & 4l+4 & \underline{2l+3} & \underline{4l+8} \\
& \uparrow & & & & & & & & & \\
& M_{2i+1,p} & 0 & 0 & 4 & 4 &\ldots & 4l & 8l+4 & 4l+4 & 8l+12 \\
& \uparrow & & & & & & & & & \\
& L_{2i+2,\widetilde{p}} & 0 & 0 & 0 & 1 & \ldots & 2l-1 & 4l & 2l+1 & 4l+4 \\
\end{array} \end{equation}

Where the numbers on the right of a module signifies $\dim_\k((-)_x)$ for\\
${x=2i-1, \ldots, 2i+2l+3}$ and an underlined number implies a lower bound for $\dim_\k$. Similarly we write $\overline{N}$ to denote an upperbound for a certain $\dim_\k$.\\ \\
Now consider the module $P_{p,\geq 2i+2l}$. It is generated in degree $2i+2l$ because $P_p$ is a quotient of $L_{2i-1,p}$. Moreover $\dim_\k((P_p)_{2i+2l})=1$ and $\dim_\k((P_p)_{2i+2l+1})=2$, so Lemma \ref{lem:pointupperbound} implies $\dim_\k((P_p)_{2i+2l+2}) \leq 1$ and $\dim_\k((P_p)_{2i+2l+3}) \leq 2$. Together with the right exact sequence (\ref{eq:pointresolution}) this gives us the following upper bounds:
\begin{equation} \label{eq:secondbound} \begin{array}{cccccccccc}
&0& & & & & & & & \\
& \uparrow & & & & & & & & \\
& P & 1 & 1 & 2 & \ldots & 1 & 2 & \overline{1} & \overline{2} \\
& \uparrow & & & & & & & & \\
& L_{2i-1,p} & 1 & 1 & 3 & \ldots & l+1 & 2l+3 & \overline{l+2} & \overline{2l+5} \\
& \uparrow & & & & & & & & \\
& L_{2i+1,p'} & 0 & 0 & 1 & \ldots & l & 2l+1 & l+1 &2l+3 \\
\end{array} \end{equation}
Combining the bounds found in (\ref{eq:firstbound}) and (\ref{eq:secondbound}) and using (\ref{eq:locallyfreeloceven}) we have:
\begin{equation} \label{eq:thirdbound} \begin{array}{cccccccccc}
&0& & & & & & & & \\
& \uparrow & & & & & & & & \\
& L_{2i-1,p} & 1 & 1 & 4 & \ldots & l+1 & 2l+3 & \overline{l+2} & \overline{2l+5} \\
& \uparrow & & & & & & & & \\
& L_{2i,q} & 0 & 1 & 4 & \ldots & 2l+1 & 4l+4 & \underline{2l+3} & \underline{4l+8} \\
& \uparrow & & & & & & & & \\
& L_{2i+1,\widetilde{p}} & 0 & 0 & 1 & \ldots & l-1 & 2l+1 & l+1 & 2l+3 \\
\end{array} \end{equation}
Right exactness of (\ref{eq:locallyfreeloceven}) implies that the bounds in (\ref{eq:thirdbound}) are in fact equalities. By way of example we find the upper bound
\begin{eqnarray*}
 \dim_\k((L_{2i,q})_{2i+2l+2} ) & \leq & \dim_\k((L_{2i-1,p})_{2i+2l+2} ) + \dim_\k((L_{2i+1,\widetilde{p}})_{2i+2l+2} ) \\
& \leq & l+2 + l +1 \\
& = & 2l+3
\end{eqnarray*}
which equals the already known lower bound for $\dim_\k((L_{2i,q})_{2i+2l+2}$.  Hence we have found exact values for $\dim_\k(L_{2i+1,q})$. A priori the above right exact sequence only gives those exact value for the points $q \in X_{2i}$ for which there is a $p \in X_{2i-1}$ such that $\r{O}_p \tr \r{E}^{*2i-1} = \r{O}_q$. But as $\r{E}^{*2i-1}$ is of the form ${}_{\Id}(\r{L}_{i-1})_f$ as in (\ref{eq:forappendixc}) we have $q = f(p)$ and surjectivity of $f$ implies that $q$ runs through all points of $X_{2i}$ as $p$ runs through all points of $X_{2i-1}$. With the same reasoning we now obtain  the exact values
for $\dim_k (L_{2i-1})_{2i+2l+2}$ and
$\dim_k (L_{2i-1})_{2i+2l+3}$. . \\ \\
Hence we have proven (\ref{eq:correctdimensions}) for all $i,j \in \d{Z}$ and for all points $p$. As these values do not depend on $p$  and $X$ is a smooth variety, it follows from \cite[ex. II, \S 5, no.8]{Hartshorne77} that the terms in (\ref{eq:locallyfreeleft}) are locally free on the left (and hence also on the right). Filling in these values for (\ref{eq:locallyfreestalklike1}), the theorem follows.
\end{proof}

\section{The Full Exceptional Sequence} 

%We have done the preparatory work needed to compute the dimensions of the $\Ext$-groups of the exceptional sequence (\ref{eq:exceptionalsequence}). We will prove that the sequence. 
This section is dedicated to the proof of the following theorem:

\begin{theorem}
\label{thm:mainexoticsequence}
	Let  $\r{E}$ be a $\d{P}^1$-bimodule of rank $(4,1)$. Let $\d{S}(\r{E})$ be the associated symmetric sheaf $\d{Z}$-algebra and put $Z=\Proj(\d{S}(\r{E}))$. Let $\D$ denote the triangulated subcategory of objects in $\D(Z)$ with bounded noetherian cohomology. Then $\D$ is $\Ext$-finite and
\begin{equation} \label{eq:fullstrongex}
\bigg( \Pi_1^*(\r{O}_{\d{P}^1}),\Pi_1^*(\r{O}_{\d{P}^1}(1)), \Pi_0^*(\r{O}_{\d{P}^1}),\Pi_0^*(\r{O}_{\d{P}^1}(1)\bigg)
\end{equation}
is a full strong exceptional sequence in $\D$. In the particular case where $\r{E} = {}_f(\mathcal{O})_{Id}$ for a morphism $f:\d{P}^1\mor \d{P}^1$ of degree 4, the Gram matrix of the Euler form for this exceptional sequence is given by
\[
\begin{bmatrix}
1&2&1&5\\
0&1&0&4\\
0&0&1&2\\
0&0&0&1\\
\end{bmatrix}
\]
\end{theorem}

We will prove this theorem through a series of lemmas. We first exhibit some technical results required to show that the sequence is indeed full
\begin{lemma}
\label{lem:crit}
Let $\r{T}$ be a $\k$-linear  triangulated category. 
Assume that $E_1,\ldots,E_n$ is a collection of objects in $\r{T}$ such that
\begin{itemize}
\item[(a)]
$
 \sum_j \dim\Hom_\r{T}^j(E_i,T)<\infty\text{ for all $i$ and for all $T\in \Ob(\r{T})$}
$.
\item[(b)]
 $(E_i)_i$ satisfies the conditions for an exceptional sequence,
except that we do not require $\Hom$-finiteness of $\r{T}$.
\item[(c)] we have 
\[ ((E_i)_i)^{\perp}:=\{Y\in \T\,\vert \ \forall m: Hom^i(E_m,Y)=0 \}= \{0\} \]
\end{itemize}
Then 
\begin{enumerate}
\item $E_1,\ldots,E_n$ generate $\r{T}$ as a triangulated category and
\item $\r{T}$ is $\Ext$-finite.
\end{enumerate}
\end{lemma}
\begin{proof} Let $T\in \r{T}$. We have to prove that $T$ is in the triangulated subcategory of $\r{T}$
generated by $E_1,\ldots,E_n$. We put $T_n=T$ and  define $T_{i-1}$ inductively by $L_{E_i}T_i$ for $i=n,n-1,\ldots, 1$, i.e.
\[
T_{i-1}=\cone(\Hom_\r{T}^\bullet(E_i,T_{i})\otimes_k E_i\mor T_{i})
\]
Then $T_i$ is in the triangulated subcategory of $\r{T}$ generated by $T_{i-1}$ and $E_i$. Furthermore
\[
\Hom_{\r{T}}^\bullet(E_j,T_i)=0\text{ for } j>i
\]
It follows that $T_0=0$. Hence we are done. 

For (2) we have to prove that if $T_1,T_2\in \r{T}$ then $\sum_j \dim\Hom^j_{\r{T}}(T_1,T_2)<\infty$. Since $T_1$ is in
the triangulated category generated by $(E_i)_i$ we may assume $T_1=E_i$ for some $i$. But then the claim is 
part of the hypotheses.
\end{proof}
\begin{lemma} 
\label{lem:tau} Let $X,Y$ be smooth varieties over $\k$ and let $\r{E}$ be a locally free $X-Y$-bimodule of rank $(4,1)$. Then for all $m\in \Z$ one has
\begin{enumerate}
\item The cohomological dimension of $\Pi_{m,*}$ satisfies\footnote{It is easy to see that in (1) the cohomological dimension is exactly one, but we do 
not need it and leave it out for clarity} 
\[
\cd \Pi_{m,\ast}\le 1.
\]
\item If $\r{F}$ is a noetherian object then $R^i\Pi_{m,\ast}\r{F}$ is a coherent sheaf for all $i$.
\end{enumerate}
\end{lemma}
\begin{proof} 
\begin{enumerate}
\item By (\ref{eq:pushforward/tau}), we have
\[
\der^i\Pi_{m,\ast}(p(-))=\der^i\omega(p(-))_m
\]
which reduces the claim to $\cd \omega=1$. From Lemma \ref{lem:tau/omega} we  in turn obtain
\[
\der^i\omega(p(-)) \cong \der^{i+1}\tau 
\]

and the claim now reduces to $\cd \tau=2$. This is proved as in \cite[Corollary 4.10]{Nyman04} using the exact sequence (\ref{eq:locallyfree})
instead of the exact sequence (4.1) in loc.\ cit.
\item Since $\Gr(\d{S}(\r{E}))$ is locally noetherian we may construct a left resolution of $\r{F}$ by
objects which are finite direct sums of objects of the form 
\[
p(\r{G}\otimes_{\r{O}_{X_n}} e_n \d{S}(\r{E}))=\Pi^\ast_n(\r{G})
\]
for $\r{G}\in \coh(X_n)$. Using that $\Pi_{m,\ast}$ has finite cohomological dimension we reduce to the case $\r{F}=\Pi^\ast_n(\r{G})$.

Tensoring (\ref{eq:locallyfree}) (with $m$ replaced by $n$) on the left with $\r{G}\in \coh(X_{n})$ we obtain exact sequences
in $Z=\Proj (\d{S}(\r{E}))$
\begin{equation}
\label{eq:basic}
0\mor \Pi^{\ast}_{n+2}(\r{G})\mor \Pi^{\ast}_{n+1}(\r{G}\tr_{X_{n}} \r{E}^{\ast n}) \mor \Pi^\ast_n(\r{G})\mor 0
\end{equation}
Hence repeatedly using such exact sequences we may reduce
to the case $\r{F}=\Pi^\ast_n(\r{G})$ for $n\le m$.
When $n\le m$ it is shown in the proof of theorem 5.4.1 that
\[
R^i\Pi_{m,\ast} \Pi^\ast_n \r{G}=
\begin{cases}
\r{G}\otimes_{X_n} \d{S}(\r{E})_{n,m}&\text{if $i=0$}\\
0&\text{otherwise}
\end{cases}
\]
\end{enumerate}
This is indeed coherent. 
\end{proof}

\begin{lemma} 
\label{lem:ort} Assume $X=\d{P}^1$ and let $f:Y\mor X$ be a morphism of degree 4. Put $\r{E}={}_f(\r{O}_X)_{\Id}$. Then the right orthogonal to the subcategory generated by
\[
E=(\Pi_{1}^\ast(\r{O}_{\d{P}^1}),
\Pi_{1}^\ast(\r{O}_{\d{P}^1}(1)),
\Pi_{0}^\ast(\r{O}_{\d{P}^1}),
\Pi_{0}^\ast(\r{O}_{\d{P}^1}(1)))
\]
in $\D(\Proj (\d{S}(\r{E})))$ is zero.
\end{lemma}
\begin{proof} 
Assume that $A\in \Proj (\d{S}(\r{E}))$ is right orthogonal to $E$.
Using the exact
sequences
\[
0\mor \r{O}_{\d{P}^1}(a)\mor \r{O}_{\d{P}^1}(a+1)^{\oplus 2}\mor \r{O}_{\d{P}^1}(a+2)\mor 0
\]
and the exactness of $\Pi^\ast_m$ (Lemma \ref{lem:pushforwardexact}) we find that $A$ is
right orthogonal to $\Pi_{m}^\ast(\r{O}_{\d{P}^1}(a))$ for $m=0,1$ and all $a$.

From (\ref{eq:basic}) we obtain exact sequences in $\Proj (\d{S}(\r{E}))$
\[
0\mor \Pi^{\ast}_{m+2}(\r{O}_{\d{P}^1}(a))\mor \Pi^{\ast}_{m+1}(\r{O}_{\d{P}^1}(a)\otimes_{X_{m}} \tr {E}^{\ast m})
\mor \Pi^\ast_m(\r{O}_{\d{P}^1}(a))\mor 0
\]
Since $\r{O}_{\d{P}^1}(a)\otimes_{X_{m}} \r{E}^{\ast m}$, being locally free, is isomorphic to a sum of $\r{O}_{\d{P}}(b)$ 
we conclude by induction that $A$ is right orthogonal to $\Pi_{m}^\ast(\r{O}_{\d{P}^1}(a))$ for all $m,a$.

Now $(\Pi_{m}^\ast(\r{O}_{\d{P}^1}(a)))_{m,a}$ is a collection of generators for $\Proj (\d{S}(\r{E}))$ as a Grothendieck category. From
this it is easy to see that the right orthogonal to $(\Pi_{m}^\ast(\r{O}_{\d{P}^1}(a)))_{m,a}$ in $\D(\Proj (\d{S}(\r{E})))$
is zero. This finishes the proof.
\end{proof}

\begin{proof}\emph{of \ref{thm:mainexoticsequence}}.
The computation of the Gram matrix, the strongness and exceptionality  is an immediate application of the formula \ref{thm:extandpullback}:
\[
\Ext^i_Z \left( \Pi_n^*\r{F}, \Pi_n^*\r{G} \right)=\Ext^i_{\d{P}^1} \left( \r{F}, \r{G} \tr \d{S}(\r{E})_{n,n} \right)=\Ext^i_{\d{P}^1} \left( \r{F}, \r{G}  \right)
\]
proving the claim for the subsequences
\begin{center}
$\big(\Pi_1^*(\r{O}_{\d{P}^1}),\Pi_1^*(\r{O}_{\d{P}^1}(1))\big)$ and $\big(\Pi_0^*(\r{O}_{\d{P}^1}),\Pi_0^*(\r{O}_{\d{P}^1}(1)\big)$
\end{center}
There are no backward $\Hom$'s by the formula \ref{thm:extandpullback} once again.\\
There are four remaining cases. Since they are all very similar, we pick one out and leave the other three to the reader:
\begin{align*}
	\Ext^i_Z \left( \Pi_1^*( \r{O}_{\d{P}^1}), \Pi_0^*(\r{O}_{\d{P}^1}(1) \right)
	=&\Ext^i_{\d{P}^1}(\r{O}_{\d{P}^1},\r{O}_{\d{P}^1}(1)\tr \d{S}(\r{E})_{0,1})\\
	=&\Ext^i_{\d{P}^1}(\r{O}_{\d{P}^1},\r{O}_{\d{P}^1}(1)\tr {}_f (\r{O}_{\d{P}^1})_{\Id} )\\
	=&\H^i(\d{P}^1,\r{O}_{\d{P}^1}(1)\tr \r{O}_{\d{P}^1}(1)\tr {}_f (\r{O}_{\d{P}^1})_{\Id})\\
	\stackrel{(\ref{eq:pullbacktensorproduct})}{=}& \H^i(\d{P}^1,f^*\r{O}_{\d{P}^1}(1))\\
	=& \H^i(\d{P}^1,\r{O}_{\d{P}^1}(4))\
\end{align*}
which is indeed only nonzero for $i\neq 0$, in which case it is 5-dimensional over $\k$.\\[\medskipamount]
To show that the sequence is full, we have to verify conditions (a)(b)(c) of lemma \ref{lem:crit}. Condition (a) follows from Lemma \ref{lem:tau},  which implies that
$R\Pi_{m,\ast}\r{G}$ lives in $D^b_{\coh}(\Qcoh(X_m))$, combined with 
the fact that by Lemma \ref{cor:differentRHom}, we have
\[ \Ext^i_{\Proj \r{A}}(\Pi^\ast_m(\r{O}_{\d{P}^1}(a)),\r{G})=\Ext^i_{X_m}(\r{O}_{\d{P}^1}(a),R\Pi_{m,\ast}\r{G})\]
Condition (b) is proven above.
Finally, condition (c) follows from Lemma \ref{lem:ort}.
\end{proof}

As the exceptional collection (\ref{eq:fullstrongex}) is full and strong we can prove the following:
\begin{theorem}
Let $\r{E} = {}_f \left( \mathcal{O}_{\mathbb{P}^1} \right)_{Id}$ and $\D$ be as in Theorem \ref{thm:mainexoticsequence}, then there is an equivalence
\[ \D \cong \D(\k Q/I) \]
where $Q$ is the quiver
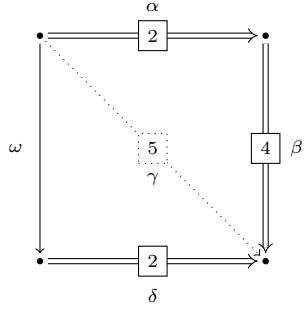
\begin{center}
\begin{tikzpicture}[
    implies/.style={double,double equal sign distance,-implies},
    dot/.style={shape=circle,fill=black,minimum size=2pt,
                inner sep=0pt,outer sep=2pt}]
\node[vertice,dot] (a) at ( -1.5, 0) {};
\node[vertice,dot] (b) at (1.5, 0) {};
\node[vertice,dot] (c) at (1.5, -3) {};
\node[vertice,dot] (d) at ( -1.5, -3) {};
\tikzset{every node/.style={fill=white},rectangle}
\draw[->,font=\scriptsize]
    (a) edge[dotted] node[vertice,rectangle]{$5$} node[below=0.2cm]{$\gamma$} (c)
    (d) edge[implies] node[vertice,rectangle]{$2$} node[below=0.25cm]{$\delta$} (c)
    (a) edge node[left=0.1cm]{$\omega$} (d)
    (a) edge[implies] node[vertice,rectangle]{$2$} node[above=0.2cm]{$\alpha$} (b)
    (b) edge[implies] node[vertice,rectangle]{$4$} node[right=0.2cm]{$\beta$} (c);
\end{tikzpicture}
\end{center}
and the relations $I$ are such that
\begin{itemize}
\item There is a 5-dimensional space of ``diagonal'' morphisms $\gamma_n$.
\item Each $\gamma_n$ can be written as a linear combination of the $\alpha_i \beta_j$. 
\item $\omega \delta_0$ and $\omega \delta_1$ are linearly independent.
\end{itemize}
\end{theorem}
\begin{proof}
As the exceptional collection (\ref{eq:fullstrongex}) is full and strong, there is a tilting object
\[ \r{T} := \Pi_1^*(\r{O}) \oplus \Pi_1^*(\r{O}(1)) \oplus \Pi_0^*(\r{O}(1)) \oplus \Pi_0^*(\r{O}) \]
showing that $\D \cong \D(\End(\r{T}))$. Now $\End(\r{T})$ is obviously isomorphic to $kQ/J$ where $Q$ is given by the quiver
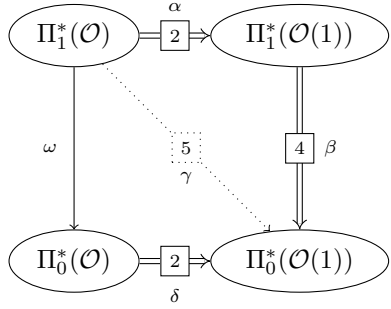
\begin{center}
\begin{tikzpicture}[
    implies/.style={double,double equal sign distance,-implies},
    dot/.style={shape=circle,fill=black,minimum size=2pt,
                inner sep=0pt,outer sep=2pt}]
\node[vertice,ellipse] (a) at ( -1.5, 0) {$\Pi_1^*(\r{O})$};
\node[vertice,ellipse] (b) at (1.5, 0) {$\Pi_1^*(\r{O}(1))$};
\node[vertice,ellipse] (c) at (1.5, -3) {$\Pi_0^*(\r{O}(1))$};
\node[vertice,ellipse] (d) at ( -1.5, -3) {$\Pi_0^*(\r{O})$};
\tikzset{every node/.style={fill=white},rectangle}
\draw[->,font=\scriptsize]
    (a) edge[dotted] node[vertice,rectangle]{$5$} node[below=0.2cm]{$\gamma$} (c)
    (d) edge[implies] node[vertice,rectangle]{$2$} node[below=0.25cm]{$\delta$} (c)
    (a) edge node[left=0.1cm]{$\omega$} (d)
    (a) edge[implies] node[vertice,rectangle]{$2$} node[above=0.2cm]{$\alpha$} (b)
    (b) edge[implies] node[vertice,rectangle]{$4$} node[right=0.2cm]{$\beta$} (c);
\end{tikzpicture}
\end{center}
and the relations in $J$ are induced by composition for the $\Hom$-sets in the exceptional collection (\ref{eq:fullstrongex}). To check that relations in $J$ actually satisfy the 3 above conditions, note that the above quiver can be identified with:
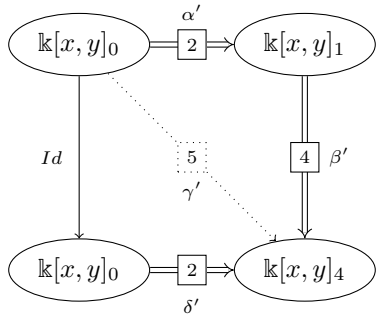
\begin{center}
\begin{tikzpicture}[
    implies/.style={double,double equal sign distance,-implies},
    dot/.style={shape=circle,fill=black,minimum size=2pt,
                inner sep=0pt,outer sep=2pt}]
\node[vertice,ellipse] (a) at ( -1.5, 0) {$\k[x,y]_0$};
\node[vertice,ellipse] (b) at (1.5, 0) {$\k[x,y]_1$};
\node[vertice,ellipse] (c) at (1.5, -3) {$\k[x,y]_4$};
\node[vertice,ellipse] (d) at ( -1.5, -3) {$\k[x,y]_0$};
\tikzset{every node/.style={fill=white},rectangle}
\draw[->,font=\scriptsize]
    (a) edge[dotted] node[vertice,rectangle]{$5$} node[below=0.2cm]{$\gamma'$} (c)
    (d) edge[implies] node[vertice,rectangle]{$2$} node[below=0.25cm]{$\delta'$} (c)
    (a) edge node[left=0.1cm]{$Id$} (d)
    (a) edge[implies] node[vertice,rectangle]{$2$} node[above=0.2cm]{$\alpha'$} (b)
    (b) edge[implies] node[vertice,rectangle]{$4$} node[right=0.2cm]{$\beta'$} (c);
\end{tikzpicture}
\end{center}
Where the $\alpha_i'$, $\beta_j'$, $\gamma_m'$ give vectorspace bases for $k[x,y]_1$, $k[x,y]_3$ and $k[x,y]_4$ respectively and $\delta'_0, \delta'_1$ are homogeneous degree 4 polynomials defining $f$:
\[ f: \mathbb{P}^1 \rightarrow \mathbb{P}^1: [x:y] \mapsto [\delta_0'(x:y):\delta_1'(x:y) ] \]
\end{proof}
\begin{remark}
In the special case where $f: \mathbb{P}^1 \rightarrow \mathbb{P}^1$ is given by $[x:y] \mapsto [x^4:y^4]$ the relations are given by
\begin{equation} \label{eq:tiltquiverrel} \left\{ \begin{array}{cc} \alpha_i \beta_j = \gamma_{i+j} & 0 \leq i \leq 1, 0 \leq j \leq 3 \\
\omega \delta_i = \gamma_{4i} & 0 \leq i \leq 1 \end{array} \right. \end{equation}
as in this case $\alpha_i'$, $\beta_j'$, $\gamma_m'$ and $\delta_n'$ denote multiplication by $x^i y^{1-i}$, $x^j y^{3-j}$, $x^my^{4-m}$ and $x^{4n}y^{4-4n}$ respectively.
\end{remark}

\begin{remark}
The total number of degrees of freedom is 3 (in choosing $f$, or in choosing such a quiver). This can be intuitively seen as follows:
\begin{itemize}
\item We first fix bases for the 4 vertices
\item There are $8=2 \cdot 4$ compositions of $\alpha_i \beta_j$. As these generate the 5-dimensional space of diagonal morphisms, there are 3 relations between them. The degrees of freedom for these choices is given by the dimension of $Gras(3,8)$ which is $3(8-3)=15$
\item The $\omega \delta_m$ should be expressed in the 5-dimensional space of $\gamma_n$. The amount of ways this can be done is given by the amount of morphisms from a 2-dimensional vectorspace to a 5-dimensional one: hence 10 ways.
\item Now we can base change each of the 4 vertices, giving an action of $GL(1) \times GL(2) \times GL(2) \times GL(2)$, which is 1+4+4+16=25-dimensional. But we should mod out this group by all scalar multiplications by $a, b, c, d$ respectively which satisfy $ab=cd$. So there is an action by a 22-dimensional group.
\item An action of 22-dimensional group on a 25 dimensional space gives a 3-dimensional moduli space.
\end{itemize}

\end{remark}
\newpage

\bibliographystyle{plain}
\bibliography{NC_P^1-Bundles}

\begin{thebibliography}{10}

\bibitem{ArtZhang94}
M.~Artin and J.~Zhang.
\newblock Noncommutative projective schemes.
\newblock {\em Adv. Math.}, 109(2):228--287, 1994.

\bibitem{deTVdB14}
L.~de~Thanhoffer~de Volcsey and M.~Van den Bergh.
\newblock Numerical classification of exceptional collections of length 4 on
  del pezzo surfaces.
\newblock in preparation, 2015.

\bibitem{Genpreprojective}
L.~de~Thanhoffer~de Volcsey and D.~Presotto.
\newblock {Some generalizations of Preprojective algebras and their
  properties}.
\newblock November 2014.
\newblock {arXiv 1412.6899}.

\bibitem{Hartshorne77}
R.~Hartshorne.
\newblock {\em Algebraic Geometry}.
\newblock Graduate Texts in Mathematics. Springer-Verslag, 8 edition, 1997.

\bibitem{Mori07}
Izuru Mori.
\newblock Intersection theory over quantum ruled surfaces.
\newblock {\em Journal of Pure a}, 211:25--41, nd Applied Algebra.

\bibitem{Nyman04}
A.~Nyman.
\newblock Serre duality for noncommutative $\mathbb{P}^1$-bundles.
\newblock {\em Trans. AMS.}, 357(4):1349--1416, 2004.

\bibitem{Nyman09}
A.~Nyman.
\newblock Serre finiteness and serre vanishing for noncommutative
  $\mathbb{P}^1$-bundles.
\newblock {\em Journal of Algebra}, 278(1):32--42, 2004.

\bibitem{Nyman15}
A.~Nyman.
\newblock {Noncommutative Tsen's theorem in dimension one}.
\newblock January 2015.
\newblock {arXiv 1408.3748}.

\bibitem{Smith_99}
Paul Smith.
\newblock Noncommutative algebraic geometry.
\newblock 1999.

\bibitem{VdB_12}
M.~Van~den Bergh.
\newblock Noncommutative $\mathbb{P}^1$ bundles over commutative schemes.
\newblock {\em Trans. AMS.}, 364(12):6279--6313, 2012.

\end{thebibliography}
\nocite{*}
\end{document}